\documentclass[10pt]{amsart}

\renewcommand{\div}{\operatorname{div}}
\usepackage{amssymb}
\usepackage[leqno]{amsmath}
\usepackage{mathrsfs}
\usepackage{stmaryrd}
\usepackage{chemarrow}
\usepackage[norelsize]{algorithm2e}
\usepackage{enumerate}
\usepackage{graphicx}
\usepackage[pdftex,bookmarksnumbered,bookmarksopen,colorlinks,linkcolor=red,anchorcolor=black,citecolor=blue,urlcolor=blue]{hyperref}
\usepackage[all]{xy}
\usepackage{tikz}
\usetikzlibrary{arrows}

\usepackage{multirow}

  \newcounter{mnote}
  \let\oldmarginpar\marginpar
    \renewcommand\marginpar[1]{\-\oldmarginpar[\raggedleft\footnotesize #1]%
    {\raggedright\footnotesize #1}}

\newtheorem{theorem}{Theorem}[section]
\newtheorem{lemma}[theorem]{Lemma}
\newtheorem{corollary}[theorem]{Corollary}

\newtheorem{remark}[theorem]{Remark}

\newcommand{\dx}{\,{\rm d}x}
\newcommand{\dd}{\,{\rm d}}
\newcommand{\bs}{\boldsymbol}
\newcommand{\mcal}{\mathcal}

\numberwithin{equation}{section}

\begin{document}
\title[FASP for linear elasticity]{Fast Auxiliary Space Preconditioner for Linear Elasticity in Mixed Form}
\author{Long Chen}%
\address{Department of Mathematics, University of California at Irvine, Irvine, CA 92697, USA}%
\email{chenlong@math.uci.edu}%
\author{Jun Hu}%
\address{LMAM and School of Mathematical Sciences, Peking University, Beijing 100871, China}%
\email{hujun@math.pku.edu.cn}%
\author{Xuehai Huang$^{\ast}$}%
\thanks{$^{\ast}$Corresponding author.}
\address{College of Mathematics and Information Science, Wenzhou University, Wenzhou 325035, China}%
\email{xuehaihuang@wzu.edu.cn}%

\thanks{The first author was supported by  NSF Grant DMS-1418934. This work was finished when L. Chen visited Peking University in the fall of 2015. He would like to thank Peking University for the support and hospitality, as well as for their exciting research atmosphere.}
\thanks{The second author was supported by  the NSFC Projects 11271035,  91430213 and 11421101.}
\thanks{The third author was supported by the NSFC Project 11301396, and Zhejiang Provincial
Natural Science Foundation of China Projects LY14A010020, LY15A010015 and LY15A010016.}


\begin{abstract}
A block diagonal preconditioner with the minimal residual method and a block triangular preconditioner with the generalized minimal residual method are developed for Hu-Zhang mixed finite element methods of linear elasticity. They are based on a new stability result of the saddle point system in mesh-dependent norms. The mesh-dependent norm for the stress corresponds to the mass matrix which is easy to invert while the displacement it is spectral equivalent to Schur complement. A fast auxiliary space preconditioner based on the $H^1$ conforming linear element of the linear elasticity problem is then designed for solving the Schur complement. For both diagonal and triangular preconditioners, it is proved that the conditioning numbers of the preconditioned systems are bounded above by a constant independent of both the crucial Lam\'{e} constant and the mesh-size. Numerical examples are presented to support theoretical results. As  byproducts, a new stabilized low order mixed finite element method is proposed and analyzed and superconvergence results of Hu-Zhang element are obtained.
\end{abstract}
\maketitle



\section{Introduction}
We consider fast solvers for the Hu-Zhang mixed finite element methods \cite{Hu2015a,HuZhang2015c, HuZhang2015} for linear elasticity, namely fast solvers for inverting the following saddle point system
\begin{equation}\label{intro:basic}
\begin{pmatrix}
M_h^{\lambda} & B_h^T\\
B_h & O
\end{pmatrix},
\end{equation}
where $M_h^{\lambda}$ is the mass matrix weighted by the compliance tensor and $B_h$ is the discretization of the $\div$ operator. The subscript $h$ is the mesh size of a underlying triangulation and the superscript $\lambda$ is the Lam\'e number which could be very large for nearly incompressible material. We aim to develop preconditioners robust to both $h$ and $\lambda$.

In \cite{Hu2015a,HuZhang2015c, HuZhang2015}, a stability result is established in the $H(\div;\Omega)\times L^2(\Omega)$ norm whose matrix form is $(M_h+B_h^TM_{u,h}^{-1}B_h)\times M_{u,h}$, where $M_{u,h}$ is the mass matrix for the displacement and $M_h$ is the abbreviation of $M_h^0$.  By the theory developed by Mardal and Winther \cite{MardalWinther2011}, the following block diagonal preconditioner leads to a parameter independent condition number of the preconditioned system
\begin{equation}\label{Hdiv}
\begin{pmatrix}
 (M_h+B_h^TM_{u,h}^{-1}B_h)^{-1} & O\\
 O & M_{u,h}^{-1}
\end{pmatrix}.
\end{equation}
To compute the first block of \eqref{Hdiv}, however, a non-trivial solver should be designed to account for the discrete $\div$ operator.

Motivated by our recent work \cite{chen2016multigrid}, we shall establish another stability result of \eqref{intro:basic} in mesh dependent norms $\|\cdot\|_{0,h}\times |\cdot|_{1,h}$ whose equivalent matrix form is $M_h \times B_hM_h^{-1}B_h^T$. Therefore we can use the block diagonal preconditioner
\begin{equation}\label{intro:Schur}
\begin{pmatrix}
 M_h^{-1} & O\\
 O & (B_hM_h^{-1}B_h^T)^{-1}
\end{pmatrix}
\end{equation}
together with the MINRES method to solve \eqref{intro:basic}. The mass matrix $M_h$ can be further replaced by its diagonal matrix and thus a spectral equivalent approximation of $M_h^{-1}$ is easy to construct. The difficulty is the inverse of the Schur complement.

We shall develop a fast auxiliary space preconditioner for the Schur complement. The auxiliary space preconditioner was initially designed by Xu~\cite{Xu1996} to avoid the difficulty in creating a sequence of nonnested grids or nonnested finite element spaces. As a two level method, the auxiliary space preconditioner involves smoothing on the fine level space which is usually the to-be-solved finite element space, and a coarse grid correction on an auxiliary space which is much more flexible to choose. It has been successfully applied to many finite element methods for partial differential equations, including conforming and nonconforming finite element method for the second order or fourth order problem~\cite{Xu1996, ZhangXu2014}, $H(\textrm{curl})$ and $H(\textrm{div})$ problems~\cite{HiptmairXu2007, KolevVassilevski2008, KolevVassilevski2009,TuminaroXuZhu2009, KrausLazarovLymberyMargenovEtAl2016}, DG type discretizations~\cite{BrixCamposDahmen2008, CockburnDuboisGopalakrishnanTan2014,ChenWangWangYe2015, LiXie2015, ZhongChungLiu2015}, and general
symmetric positive definite problems \cite{KrausLymberyMargenov2015} etc.

We use the $H^1$ conforming linear finite element discretization on the same mesh for the linear elasticity equation with parameter $\lambda = 0$ as the auxiliary problem to preconditioning the Schur complement. Since $\lambda = 0$, we can solve the auxiliary problem by geometric multigrid methods for structured meshes and algebraic multigrid methods in general. Using the Korn's inequality, we can further adopt the $H^1$ conforming linear finite element discretization for vector-type Poisson equation as the auxiliary problem.


Our stability result is robust to the parameter $\lambda$ and $h$, therefore the condition number of the preconditioned system is uniformly bounded with respect to both the size of the problem and the parameter $\lambda$. The later is notoriously difficult to construct for linear elasticity. Furthermore our results hold without the full regularity assumption.

We now give a brief literature review on robust multigrid methods for the linear elasticity problem.
Discretization of the linear elasticity equations can be classified into three categories: displacement primary formulation, displacement-pressure mixed formulation and stress-displacement mixed formulation.
Robust conforming and nonconforming multigrid methods for the primary formulation have been discussed in \cite{Schoberl1999, Wieners2000, LeeWuChen2009}, and  discontinuous Galerkin $H(\div)$-conforming method in \cite{HongKrausXuZikatanov2016}.
The W-cycle multigrid methods are the most studied multigrid methods for the displacement-pressure mixed formulation, which can be found in \cite{Lee1998, BrennerLiSung2014} for conforming discretization and \cite{Brenner1993, Brenner1994} for nonconforming discretization. A V-cycle multigrid method for the finite difference discretization was developed in \cite{ZhuSifakisTeranBrandt2010}.
In \cite{AxelssonPadiy1999}, the Taylor-Hood element method was reduced to the pressure Schur complement equation, based on which an inner/outer iteration scheme was set up.
So far the solvers for the stress-displacement mixed formulation are mainly concentrated on the block diagonal preconditioned MINRES method, see \cite{KlawonnStarke2004, Wang2005, PasciakWang2006}. In \cite{KlawonnStarke2004}, the multigrid preconditioner was advanced for the PEERS element method with weakly symmetric stress.
As for the Arnold-Winther element discretization,
the overlapping Schwarz preconditioner was exploited in \cite{Wang2005}, and the variable V-cycle multigrid preconditioner was developed in \cite{PasciakWang2006}.
The majority of existing works is to deal with the discrete null space ker($\div$) by using either block-wise Gauss-Seidels smoother or overlapping Schwarz smoothers.
And only works in \cite{LeeWuChen2009, HongKrausXuZikatanov2016, BrennerLiSung2014, KlawonnStarke2004, PasciakWang2006} do not rely on the $H^2$ regularity assumption.
As we mentioned early our approach do not require a prior knowledge of the discrete ker($\div$). We transfer this difficulty to solve the Schur complement problem but with $\lambda=0$, which only involves standard Poisson-type solvers. So it is much easier to implement and analyze.

To further improve the performance, we propose the following block-triangular preconditioner
\begin{equation}
\begin{pmatrix}
 I   &  D_h^{-1} B_h^T \\
0  &   -I
\end{pmatrix}
\begin{pmatrix}
D_h    &   0  \\
B_h   &  \tilde S_h
\end{pmatrix}^{-1},
\end{equation}
where $D_h$ is the diagonal of $M_h$ and $\tilde S_h = B_hD_h^{-1}B_h^T$ will be further preconditioned by the auxiliary space preconditioner we mentioned before. Numerical results in Section 6 show that the preconditioned GMRES converges around $40$ steps to push the relative tolerance below $10^{-8}$.

Results in this paper can be also applied to other $H(\div)$ conforming and symmetric stress elements developed in \cite{ArnoldWinther2002,ArnoldAwanou2005,AdamsCockburn2005,ArnoldAwanouWinther2008}. Indeed we present our results for both the original Hu-Zhang element $k\geq n+1$ and a new stabilized version for $1\leq k\leq n$.

%


The rest of this article is organized as follows. In Section 2, we present the mixed finite element methods for linear elasticity. In Section 3, we establish the stability based on the mesh dependent norms. Then we describe the block diagonal and triangular preconditioners in Section 4 and construct an auxiliary space preconditioner in Section 5.
In Section 6, we give some numerical experiments to demonstrate the efficiency and robustness of our preconditioners. Throughout this paper, we use ``$\lesssim\cdots $" to mean that ``$\leq C\cdots$", where $C$ is a generic positive constant independent of $h$ and the Lam$\acute{e}$ constant $\lambda$, which may take different values at different appearances.


\section{Mixed Finite Element Methods}

Assume that $\Omega\subset \mathbb{R}^n$ is a bounded
polytope. Denote by $\mathbb{S}$ the space of all symmetric $n\times n$ tensors.
Given a bounded domain $G\subset\mathbb{R}^{n}$ and a
non-negative integer $m$, let $H^m(G)$ be the usual Sobolev space of functions
on $G$, and $\boldsymbol{H}^m(G; \mathbb{X})$ be the usual Sobolev space of functions taking values in the finite-dimensional vector space $\mathbb{X}$ for $\mathbb{X}$ being $\mathbb{S}$ or $\mathbb{R}^n$. The corresponding norm and semi-norm are denoted respectively by
$\Vert\cdot\Vert_{m,G}$ and $|\cdot|_{m,G}$. Let $(\cdot, \cdot)_G$ be the standard inner product on $L^2(G)$ or $\boldsymbol{L}^2(G; \mathbb{X})$. If $G$ is $\Omega$, we abbreviate
$\Vert\cdot\Vert_{m,G}$, $|\cdot|_{m,G}$ and $(\cdot, \cdot)_G$ by $\Vert\cdot\Vert_{m}$, $|\cdot|_{m}$ and $(\cdot, \cdot)$,
respectively. Let $\boldsymbol{H}_0^m(G; \mathbb{R}^n)$ be the closure of $\boldsymbol{C}_{0}^{\infty}(G; \mathbb{R}^n)$ with
respect to the norm $\Vert\cdot\Vert_{m,G}$.
Denote by $\boldsymbol{H}(\mathbf{div}, G; \mathbb{S})$ the Sobolev space of square-integrable symmetric tensor fields with square-integrable divergence.
For any $\boldsymbol{\tau}\in\boldsymbol{H}(\mathbf{div}, \Omega; \mathbb{S})$, we equip the following norm
\[
\|\boldsymbol{\tau}\|_{\boldsymbol{H}(\mathbf{div})}:= \left (\|\boldsymbol{\tau}\|_0^2+\|\mathbf{div}\boldsymbol{\tau}\|_0^2\right )^{1/2}.
\]

The Hellinger-Reissner mixed formulation of the linear elasticity under the load $\boldsymbol{f}\in\boldsymbol{L}^2(\Omega; \mathbb{R}^n)$ is given as follows:
Find $(\boldsymbol{\sigma}, \boldsymbol{u})\in \boldsymbol{\Sigma}\times \boldsymbol{V}:=\boldsymbol{H}(\mathbf{div}, \Omega; \mathbb{S})\times\boldsymbol{L}^2(\Omega; \mathbb{R}^n)$ such that
\begin{align}
a(\boldsymbol{\sigma},\boldsymbol{\tau})+b(\boldsymbol{\tau}, \boldsymbol{u})&=0 \quad\quad\quad\quad\;\;\, \forall\,\boldsymbol{\tau}\in\boldsymbol{\Sigma}, \label{mixedform1} \\
b(\boldsymbol{\sigma}, \boldsymbol{v})&=-(\boldsymbol{f}, \boldsymbol{v}) \quad\quad \forall\,\boldsymbol{v}\in \boldsymbol{V}, \label{mixedform2}
\end{align}
where
\[
a(\boldsymbol{\sigma},\boldsymbol{\tau}):=(\mathfrak{A}\boldsymbol{\sigma}, \boldsymbol{\tau}),\quad
b(\boldsymbol{\tau}, \boldsymbol{v}):= (\mathbf{div}\boldsymbol{\tau}, \boldsymbol{v})
\]
with $\mathfrak{A}$ being the compliance tensor of fourth order defined
by
\[
\mathfrak{A}\boldsymbol{\sigma}:=\frac{1}{2\mu}\left(\boldsymbol{\sigma}
-\frac{\lambda}{n\lambda+2\mu}(\textrm{tr}\boldsymbol{\sigma})\boldsymbol{\delta}\right).
\]
Here $\boldsymbol{\delta}:=(\delta_{ij})_{n\times n}$ is the
Kronecker tensor, tr is the trace operator, and positive constants $\lambda$ and $\mu$ are the
Lam$\acute{e}$ constants.

Suppose the domain $\Omega$ is subdivided by a family of shape regular simplicial grids $\mathcal {T}_h$  (cf.\ \cite{BrennerScott2008,Ciarlet1978}) with $h:=\max\limits_{K\in \mathcal {T}_h}h_K$
and $h_K:=\mbox{diam}(K)$.
Let $\mathcal{F}_h$ be the union of all $n-1$ dimensional faces
of $\mathcal {T}_h$ and $\mathcal{F}_h^i$ be the union of all $n-1$ dimensional  interior faces. For any $F\in\mathcal{F}_h$,
denote by $h_F$ its diameter and fix a unit normal vector $\boldsymbol{\nu}_F$.
Let $P_m(G)$ stand for the set of all
polynomials in $G$ with the total degree no more than $m$, and $\boldsymbol{P}_m(G; \mathbb{X})$ denote the tensor or vector version of $P_m(G)$ for $\mathbb{X}$ being $\mathbb{S}$ or $\mathbb{R}^n$, respectively.

Consider two adjacent simplices $K^+$ and $K^-$ sharing an interior face $F$.
Denote by $\boldsymbol{\nu}^+$ and $\boldsymbol{\nu}^-$ the unit outward normals
to the common face $F$ of the simplices $K^+$ and $K^-$, respectively.  For a
vector-valued function $\boldsymbol{w}$, write $\boldsymbol{w}^+:=\boldsymbol{w}|_{K^+}$ and $\boldsymbol{w}^-:=\boldsymbol{w}|_{K^-}$.
Then define a  jump as
\[
[\boldsymbol{w}]:=\left\{
\begin{array}{ll}
\boldsymbol{w}^+(\boldsymbol{\nu}^+\cdot\boldsymbol{\nu}_F)+\boldsymbol{w}^-(\boldsymbol{\nu}^-\cdot\boldsymbol{\nu}_F), & \textrm{ if } F\in \mathcal{F}_h^i, \\
\boldsymbol{w}, &  \textrm{ if } F\in \mathcal{F}_h\backslash\mathcal{F}_h^i.
\end{array}
\right.
\]

For each $K\in\mathcal{T}_h$, define an $\boldsymbol{H}(\mathbf{div}, K; \mathbb{S})$ bubble function space of polynomials of degree $k$ as
\[
\boldsymbol{B}_{K,k}:=\left\{\boldsymbol{\tau}\in \boldsymbol{P}_{k}(K; \mathbb{S}): \boldsymbol{\tau}\boldsymbol{\nu}|_{\partial K}=\boldsymbol{0}\right\}.
\]
It is easy to check that $\boldsymbol{B}_{K,1}$ is merely the zero space.
Denote the vertices of simplex $K$ by $\boldsymbol{x}_{K,0}, \cdots, \boldsymbol{x}_{K,n}$.
If not causing confusion, we will abbreviate $\boldsymbol{x}_{K,i}$ as $\boldsymbol{x}_{i}$ for $i=0,\cdots, n$.
For any edge $\boldsymbol{x}_i\boldsymbol{x}_j$($i\neq j$) of element $K$, let $\boldsymbol{t}_{i,j}$ be the associated unit tangent vectors and
\[
\boldsymbol{T}_{i,j}:=\boldsymbol{t}_{i,j}\boldsymbol{t}_{i,j}^T,\quad 0\leq i<j\leq n.
\]
It has been proved in \cite{Hu2015a} that the $(n+1)n/2$ symmetric tensors $\boldsymbol{T}_{i,j}$ form a basis of $\mathbb{S}$, and for $k\geq 2$,
\[
\boldsymbol{B}_{K,k}=\sum_{0\leq i<j\leq n}\lambda_i\lambda_jP_{k-2}(K)\boldsymbol{T}_{i,j},
\]
where $\lambda_i$ is the associated barycentric coordinates corresponding to $\boldsymbol{x}_i$ for $i=0,\cdots,n$.
Some global finite element spaces are given by
\begin{align*}
\boldsymbol{B}_{k,h}&:=\left\{\boldsymbol{\tau}\in\boldsymbol{H}(\mathbf{div}, \Omega; \mathbb{S}):
\boldsymbol{\tau}|_K\in \boldsymbol{B}_{K,k} \quad \forall\,K\in\mathcal{T}_h \right\}, \\
\widetilde{\boldsymbol{\Sigma}}_{k,h}&:=\left\{\boldsymbol{\tau}\in\boldsymbol{H}^1(\Omega; \mathbb{S}):
\boldsymbol{\tau}|_K\in \boldsymbol{P}_{k}(K; \mathbb{S}) \quad \forall\,K\in\mathcal{T}_h \right\}, \\
\boldsymbol{\Sigma}_{h}&:=\widetilde{\boldsymbol{\Sigma}}_{k,h} + \boldsymbol{B}_{k,h},\label{spaceprop}\\
\boldsymbol{V}_{h}&:=\left\{\boldsymbol{v}\in \boldsymbol{L}^2(\Omega; \mathbb{R}^n): \boldsymbol{v}|_K\in \boldsymbol{P}_{k-1}(K; \mathbb{R}^n)\quad \forall\,K\in\mathcal
{T}_h\right\},
\end{align*}
with integer $ k\geq 1$. The local rigid motion space is defined as
\[
\boldsymbol{R}(K):= \left\{\boldsymbol{v}\in \boldsymbol{H}^1(K; \mathbb{R}^n): \boldsymbol{\varepsilon}(\boldsymbol{v})=\boldsymbol{0}\right\}
\]
with
$\boldsymbol{\varepsilon}(\boldsymbol{v}):=\left(\boldsymbol{\nabla}\boldsymbol{v}+(\boldsymbol{\nabla}\boldsymbol{v})^T\right)/2$ being the linearized strain tensor.

With previous preparation,
the mixed finite element method for linear elasticity proposed in \cite{Hu2015a,HuZhang2015c,HuZhang2015,ChenHuHuang2015} is defined as follows:
Find $(\boldsymbol{\sigma}_h, \boldsymbol{u}_h)\in \boldsymbol{\Sigma}_{h}\times \boldsymbol{V}_{h}$ such that
\begin{align}
a(\boldsymbol{\sigma}_h,\boldsymbol{\tau}_h)+b(\boldsymbol{\tau}_h, \boldsymbol{u}_h)&=0 \quad\quad\quad\quad\quad\, \forall\,\boldsymbol{\tau}_h\in\boldsymbol{\Sigma}_{h}, \label{stabmfem1} \\
b(\boldsymbol{\sigma}_h, \boldsymbol{v}_h) - c(\boldsymbol{u}_h, \boldsymbol{v}_h)&=-(\boldsymbol{f}, \boldsymbol{v}_h) \quad\quad \forall\,\boldsymbol{v}_h\in \boldsymbol{V}_{h}, \label{stabmfem2}
\end{align}
where
\[
c(\boldsymbol{u}_h,\boldsymbol{v}_h):=\eta\sum\limits_{F\in\mathcal{F}_h}h_{F}^{-1}\int_{F}[\boldsymbol{u}_h]\cdot[\boldsymbol{v}_h]\,{\rm d}s,
\]
\[
\eta:=\left\{
\begin{array}{ll}
0, & \textrm{ if } k\geq n+1, \\
1, &  \textrm{ if } 1\leq k\leq n.
\end{array}
\right.
\]
The bilinear form $c(\cdot,\cdot)$ involving the jump of displacement is introduced to stabilize the discretization which is only necessary for low order polynomials, i.e., $1\leq k\leq n$. Note that the scaling $h_F^{-1}$ is different with the one in \cite{ChenHuHuang2015}.

Choosing appropriate bases of $ \boldsymbol{\Sigma}_{h}$ and $\boldsymbol{V}_{h}$, we can write the matrix form of \eqref{stabmfem1}-\eqref{stabmfem2} as
\begin{equation}\label{basic}
\begin{pmatrix}
M_h^{\lambda} & B_h^T\\
B_h & -C_h
\end{pmatrix}
\begin{pmatrix}
 \bs \sigma_h\\
\bs u_h
\end{pmatrix}
=
\begin{pmatrix}
0\\
\bs f
\end{pmatrix}.
\end{equation}
where $M_h^{\lambda}$ is the mass matrix weighted by the compliance tensor, $B_h$ is the discretization of the $\div$ operator, and $C_h$ corresponds to the stabilization term. Here with a slight abuse of notation, we use the same notation $\bs \sigma_h, \bs u_h$, and $\bs f$ for the vector representations of corresponding functions.

Let
\[
\hat{\boldsymbol{\Sigma}}_{h}:=\{\boldsymbol{\tau}\in\boldsymbol{\Sigma}_{h}:
\int_{\Omega}\textrm{tr}\boldsymbol{\tau}\,\dd x=0 \},
\]
\[
\mathbb A(\boldsymbol{\sigma}_h, \boldsymbol{u}_h; \boldsymbol{\tau}_h, \boldsymbol{v}_h):=a(\boldsymbol{\sigma}_h,\boldsymbol{\tau}_h)+b(\boldsymbol{\tau}_h, \boldsymbol{u}_h)+b(\boldsymbol{\sigma}_h, \boldsymbol{v}_h) - c(\boldsymbol{u}_h, \boldsymbol{v}_h).
\]
For $k\geq n+1$, the following inf-sup condition is the immediate result of (3.4)-(3.5) in \cite{Hu2015a}:
\begin{equation}\label{eq:infsup1}
\|\widetilde{\boldsymbol{\sigma}}_h\|_{\boldsymbol{H}(\mathbf{div})} + \|\widetilde{\boldsymbol{u}}_h\|_{0}\lesssim \sup_{(\boldsymbol{\tau}_h, \boldsymbol{v}_h)\in \hat{\boldsymbol{\Sigma}}_{h}\times \boldsymbol{V}_{h}}\frac{\mathbb A(\widetilde{\boldsymbol{\sigma}}_h, \widetilde{\boldsymbol{u}}_h; \boldsymbol{\tau}_h, \boldsymbol{v}_h)}{\|\boldsymbol{\tau}_h\|_{\boldsymbol{H}(\mathbf{div})}+ \|\boldsymbol{v}_h\|_{0}},
\end{equation}
for any $(\widetilde{\boldsymbol{\sigma}}_h, \widetilde{\boldsymbol{u}}_h)\in \hat{\boldsymbol{\Sigma}}_{h}\times \boldsymbol{V}_{h}$.

Thanks to the inf-sup condtion~\eqref{eq:infsup1}, the system \eqref{basic} is stable in the space $\bs \Sigma_h \times \bs V_h$ equipped with the $H(\div;\Omega)\times L^2(\Omega)$ norm which leads to a block diagonal preconditioner requiring a non-trivial solver for $(M_h+B_h^TM_{u,h}^{-1}B_h)^{-1}$. In the next section we shall establish another stability result of \eqref{basic} in mesh dependent norms which leads to a new block-diagonal preconditioner.


\section{Stability Based On Mesh Dependent Norms}

To construct a new block diagonal preconditioner, we will show that the bilinear form $\mathbb A(\cdot, \cdot; \cdot, \cdot)$ is stable on $\hat{\boldsymbol{\Sigma}}_{h}\times \boldsymbol{V}_{h}$  with mesh dependent norms.

For each $K\in\mathcal{T}_h$,
denote by $\boldsymbol{\nu}_i$ the unit outward normal vector of the $i$-th face of element $K$. For any $\boldsymbol{\tau}_h\in \boldsymbol{\Sigma}_{h}$ and $\boldsymbol{v}_h\in \boldsymbol{V}_{h}$,
define
\begin{align*}
\|\boldsymbol{\tau}_h\|_{0,h}^2 & := \|\boldsymbol{\tau}_h\|_0^2 + \sum_{F\in \mathcal F_h}h_F\|\boldsymbol{\tau}_h\boldsymbol{\nu}_F\|_{0,F}^2 \\
|\boldsymbol{v}_h|_{1,h}^2 &:= \|\boldsymbol{\varepsilon}_h (\boldsymbol{v}_h)\|_0^2 + \sum_{F\in \mathcal F_h} h_F^{-1}\|[\boldsymbol{v}_h]\|_{0,F}^2, \\
\|\boldsymbol{v}_h\|_{c}^2 &:= c(\boldsymbol{v}_h,\boldsymbol{v}_h).
\end{align*}
Here $\boldsymbol{\varepsilon}_h$ is element-wise symmetric gradient. We shall prove the stability of \eqref{basic} in the mesh dependent norms $\|\cdot\|_{0,h}\times |\cdot|_{1,h}$. The key is the following inf-sup condition: for $k\geq n+1$
\begin{equation}\label{eq:infsup2}
|\boldsymbol{v}_h|_{1,h}\lesssim \sup_{\boldsymbol{\tau}_h \in \boldsymbol{\Sigma}_{h}} \frac{b(\boldsymbol{\tau}_h, \boldsymbol{v}_h)}{\|\boldsymbol{\tau}_h\|_{0,h}}, \quad \forall \bs v_h \in \bs V_h.
\end{equation}
For low order cases $1\leq k\leq n$, in addition to a variant of the inf-sup condition, we also need a coercivity result in the null space of the div operator.

\subsection{Properties on mesh dependent norms}
We first present a different basis of the symmetric tensor space $\mathbb S$.
Inside a simplex formed by vertices $\bs x_0,\ldots, \bs x_n$, we label the face opposite to $\bs x_i$ as the $i$-th face $F_i$. For the edge $\bs x_i \bs x_j, \, i\neq j$, define
\[
\boldsymbol{N}_{i,j}:=\frac{1}{2(\boldsymbol{\nu}_i^T\boldsymbol{t}_{i,j})(\boldsymbol{\nu}_j^T\boldsymbol{t}_{i,j})}(\boldsymbol{\nu}_i\boldsymbol{\nu}_j^T+\boldsymbol{\nu}_j\boldsymbol{\nu}_i^T),\quad 0\leq i<j\leq n.
\]
Here recall that $\bs t_{ij}$ is an unit tangent vector of edge $\bs x_i \bs x_j$ and $\bs \nu_i$ is the unit outwards normal vector of face $F_i$.
Due to the shape regularity of the triangulation, it holds
\[
\boldsymbol{\nu}_i^T\boldsymbol{t}_{i,j}\eqsim 1,\quad 0\leq i<j\leq n.
\]

By direct manipulation, we have the following results about $\boldsymbol{T}_{i,j}$ and $\boldsymbol{N}_{i,j}$:
\begin{equation}\label{eq:TNprop1}
\boldsymbol{T}_{i,j}:\boldsymbol{N}_{k,l}=\delta_{ik}\delta_{jl},\quad 0\leq i<j\leq n, \; 0\leq k<l\leq n,
\end{equation}
\begin{equation}\label{eq:TNprop2}
\boldsymbol{T}_{i,j}:\boldsymbol{T}_{i,j}=1,\quad \boldsymbol{N}_{i,j}:\boldsymbol{N}_{i,j}\eqsim1,\quad 0\leq i<j\leq n.
\end{equation}
Thus the $(n+1)n/2$ symmetric tensors $\{\boldsymbol{N}_{i,j}\}$ also form a basis of $\mathbb{S}$ which is the dual to $\{\bs T_{i,j}\}$.

\begin{lemma}\label{lem:TNnorm}
For any $q_{ij}\in L^2(K)$, $0\leq i<j\leq n$, let $\boldsymbol{\tau}_1=\sum\limits_{0\leq i<j\leq n}q_{ij}\boldsymbol{T}_{i,j}$ and $\boldsymbol{\tau}_2=\sum\limits_{0\leq i<j\leq n}q_{ij}\boldsymbol{N}_{i,j}$, then it holds
\[
\|\boldsymbol{\tau}_1\|_{0,K}^2\eqsim \|\boldsymbol{\tau}_2\|_{0,K}^2\eqsim \sum_{0\leq i<j\leq n}\|q_{ij}\|_{0,K}^2.
\]
\end{lemma}
\begin{proof}
Using the Cauchy-Schwarz inequality and \eqref{eq:TNprop2}, we have
\[
\|\boldsymbol{\tau}_1\|_{0,K}^2\leq \frac{(n+1)n}{2}\sum_{0\leq i<j\leq n}\|q_{ij}\boldsymbol{T}_{i,j}\|_{0,K}^2=\frac{(n+1)n}{2}\sum_{0\leq i<j\leq n}\|q_{ij}\|_{0,K}^2,
\]
\[
\|\boldsymbol{\tau}_2\|_{0,K}^2\leq \frac{(n+1)n}{2}\sum_{0\leq i<j\leq n}\|q_{ij}\boldsymbol{N}_{i,j}\|_{0,K}^2\lesssim\frac{(n+1)n}{2}\sum_{0\leq i<j\leq n}\|q_{ij}\|_{0,K}^2.
\]
On the other side, it follows from Cauchy-Schwarz inequality and \eqref{eq:TNprop1},
\begin{align*}
\sum_{0\leq i<j\leq n}\|q_{ij}\|_{0,K}^2=&\sum_{0\leq i<j\leq n}\int_Kq_{ij}^2\,\dd x=\sum_{0\leq i<j\leq n}\sum_{0\leq k<l\leq n}\int_Kq_{ij}q_{kl}\delta_{ik}\delta_{jl}\,\dd x \\
=&\sum_{0\leq i<j\leq n}\sum_{0\leq k<l\leq n}\int_Kq_{ij}\boldsymbol{T}_{i,j}:q_{kl}\boldsymbol{N}_{k,l}\,\dd x \\
=&\int_K\boldsymbol{\tau}_1:\boldsymbol{\tau}_2\,\dd x\leq \|\boldsymbol{\tau}_1\|_{0,K}\|\boldsymbol{\tau}_2\|_{0,K}.
\end{align*}
Hence we conclude the result by combining the last three inequalities.
\end{proof}

We then embed $\bs \varepsilon_h(\bs V_h)$ into the $\boldsymbol{H}(\mathbf{div}, K; \mathbb{S})$ bubble function space.
For each element $K\in\mathcal{T}_h$, introduce a bijective connection operator $\boldsymbol{E}_K: \boldsymbol{P}_{k-2}(K; \mathbb{S})\to\boldsymbol{B}_{K,k}$ with $k\geq 2$ as follows: for any $\boldsymbol{\tau}=\sum\limits_{0\leq i<j\leq n}q_{ij}\boldsymbol{N}_{i,j}$ with $q_{ij}\in P_{k-2}(K)$, $0\leq i<j\leq n$, define
\[
\boldsymbol{E}_K\boldsymbol{\tau}:=\sum_{0\leq i<j\leq n}\lambda_i\lambda_jq_{ij}\boldsymbol{T}_{i,j}.
\]
Applying Lemma~\ref{lem:TNnorm} and the scaling argument, we get for any $\boldsymbol{\tau}\in \boldsymbol{P}_{k-2}(K; \mathbb{S})$
\begin{equation}\label{eq:EKbound1}
\|\boldsymbol{E}_K\boldsymbol{\tau}\|_{0,K}^2\eqsim \sum_{0\leq i<j\leq n}\|\lambda_i\lambda_jq_{ij}\|_{0,K}^2\eqsim \sum_{0\leq i<j\leq n}\|q_{ij}\|_{0,K}^2\eqsim \|\boldsymbol{\tau}\|_{0,K}^2,
\end{equation}
\begin{equation}\label{eq:EKbound2}
\int_K\boldsymbol{E}_K\boldsymbol{\tau}:\boldsymbol{\tau}\,\dd x=\sum_{0\leq i<j\leq n}\int_K\lambda_i\lambda_jq_{ij}^2\,\dd x\eqsim \sum_{0\leq i<j\leq n}\|q_{ij}\|_{0,K}^2\eqsim \|\boldsymbol{\tau}\|_{0,K}^2.
\end{equation}
Denote by $\boldsymbol{E}$ the elementwise global version of $\boldsymbol{E}_K$, i.e. $\boldsymbol{E}|_K:=\boldsymbol{E}_K$ for each $K\in\mathcal{T}_h$.


Third, we give an equivalent formulation of the mesh dependent norm $|\cdot |_{1,h}$.
For each $F\in\mathcal{F}_h$, denote by $\boldsymbol{\pi}_F$ the orthogonal projection operator from $\boldsymbol{L}^{2}(F; \mathbb{R}^n)$ onto $\boldsymbol{P}_{1}(F; \mathbb{R}^n)$. Define the broken $\bs H^1$ space as
$$\boldsymbol{H}^1(\mathcal{T}_h; \mathbb{R}^n):=\left\{\boldsymbol{v}\in \boldsymbol{L}^2(\Omega; \mathbb{R}^n): \boldsymbol{v}|_K\in \boldsymbol{H}^1(K; \mathbb{R}^n)\quad \forall\,K\in\mathcal
{T}_h\right\}. $$
The domain of mesh dependent norm $|\cdot |_{1,h}$ can be extended from $\bs V_h$ to $\boldsymbol{H}^1(\mathcal{T}_h; \mathbb{R}^n)$.
\begin{lemma}\label{lem:temp32} We have the norm equivalence:
 \begin{equation}\label{eq:discreteequivnorm}
|\boldsymbol{v}|_{1,h}^2\eqsim \|\boldsymbol{\varepsilon}_h(\boldsymbol{v})\|_0^2+\sum_{F\in\mathcal{F}_h}h_F^{-1}\|\boldsymbol{\pi}_F[\boldsymbol{v}]\|_{0,F}^2 \quad \forall~\boldsymbol{v}\in \boldsymbol{H}^1(\mathcal{T}_h; \mathbb{R}^n).
\end{equation}
\end{lemma}
\begin{proof}
For any element $K\in\mathcal{T}_h$, let $\boldsymbol{\pi}_K$ be an interpolation
operator from $\boldsymbol{H}_{1}(K; \mathbb{R}^n)$ onto $\boldsymbol{R}(K)$ defined by (3.1)-(3.2) in \cite{Brenner2004a}.
And let $\boldsymbol{\pi}$ be the elementwise global version of $\boldsymbol{\pi}_K$, i.e. $\boldsymbol{\pi}|_K:=\boldsymbol{\pi}_K$ for each $K\in\mathcal{T}_h$. It follows from (3.3)-(3.4) in \cite{Brenner2004a} that for any $\boldsymbol{v}\in \boldsymbol{H}^1(\mathcal{T}_h; \mathbb{R}^n)$,
\begin{align}
\sum_{F\in\mathcal{F}_h}h_F^{-1}\|[\boldsymbol{v}]-\boldsymbol{\pi}_F[\boldsymbol{v}]\|_{0,F}^2
=&\sum_{F\in\mathcal{F}_h}h_F^{-1}\|[\boldsymbol{v}-\boldsymbol{\pi}\boldsymbol{v}]-\boldsymbol{\pi}_F[\boldsymbol{v}-\boldsymbol{\pi}\boldsymbol{v}]\|_{0,F}^2 \notag \\
\leq &\sum_{F\in\mathcal{F}_h}h_F^{-1}\|[\boldsymbol{v}-\boldsymbol{\pi}\boldsymbol{v}]\|_{0,F}^2
\lesssim \|\boldsymbol{\varepsilon}_h(\boldsymbol{v})\|_0^2. \label{eq:piFestimate}
\end{align}
Then the equivalence \eqref{eq:discreteequivnorm} follows from the triangle inequality.
\end{proof}

We shall also use the following discrete Korn's inequality (cf. (1.22) in \cite{Brenner2004a} and (34) in \cite{ArnoldBrezziMarini2005})
\begin{equation}\label{eq:korn}
\|\boldsymbol{\nabla}_h\boldsymbol{v}\|_{0}^2+\|\boldsymbol{v}\|_0^2\lesssim \|\boldsymbol{\varepsilon}_h(\boldsymbol{v})\|_0^2+\sum_{F\in\mathcal{F}_h}h_F^{-1}\|\boldsymbol{\pi}_F[\boldsymbol{v}]\|_{0,F}^2 \quad \forall~\boldsymbol{v}\in \boldsymbol{H}^1(\mathcal{T}_h; \mathbb{R}^n).
\end{equation}
Together with \eqref{eq:discreteequivnorm}, we conclude $|\cdot|_{1,h}$ defines a norm on $\bs V_h$.

\subsection{inf-sup condition in mesh dependent norms}
The inf-sup condition we need is actually for the subspace $\hat{\bs \Sigma}_h$ with vanished mean trace, c.f., \eqref{eq:infsup22} below. It is obvious that inf-sup condition \eqref{eq:infsup22} implies inf-sup condition \eqref{eq:infsup2}. On the other hand, if inf-sup condition \eqref{eq:infsup2} is true, then \eqref{eq:infsup22} holds by taking $\hat{\boldsymbol{\tau}}_h=\boldsymbol{\tau}_h-(\frac{1}{n}\int_{\Omega}\textrm{tr}\boldsymbol{\tau}_h\,\dd x) \boldsymbol{\delta}$.
Therefore inf-sup conditions \eqref{eq:infsup2} and \eqref{eq:infsup22} are equivalent.
\begin{lemma}\label{lem:infsup22}
For $k\geq n+1$, we have the following inf-sup condition
\begin{equation}\label{eq:infsup22}
|\boldsymbol{v}_h|_{1,h}\lesssim \sup_{\hat{\boldsymbol{\tau}}_h \in \hat{\boldsymbol{\Sigma}}_{h}} \frac{b(\hat{\boldsymbol{\tau}}_h, \boldsymbol{v}_h)}{\|\hat{\boldsymbol{\tau}}_h\|_{0,h}},
\end{equation}
 for any $\boldsymbol{v}_h\in \boldsymbol{V}_{h}$.
\end{lemma}
\begin{proof}
Given a $\bs v_h \in \bs V_h$, we shall construct a $\hat{\bs \tau}_h \in \hat{\bs \Sigma}_h$ to verify \eqref{eq:infsup22}.

We first control the norm $\|\boldsymbol{\varepsilon}_h(\boldsymbol{v}_h)\|_0$. For any $\boldsymbol{v}_h\in \boldsymbol{V}_{h}$, take $\boldsymbol{\tau}_1=\boldsymbol{E}\boldsymbol{\varepsilon}_h(\boldsymbol{v}_h)$.
It follows from \eqref{eq:EKbound1}
\begin{equation}\label{eq:temp4}
\|\boldsymbol{\tau}_1\|_0\eqsim \|\boldsymbol{\varepsilon}_h(\boldsymbol{v}_h)\|_0.
\end{equation}
According to integration by parts and \eqref{eq:EKbound2}, there exists a constant $C_1>0$ such that
\begin{equation}\label{eq:temp2}
b(\boldsymbol{\tau}_1, \boldsymbol{v}_h)=\int_{\Omega}\boldsymbol{\tau}_1:\boldsymbol{\varepsilon}_h(\boldsymbol{v}_h)\,\dd x\geq C_1\|\boldsymbol{\varepsilon}_h(\boldsymbol{v}_h)\|_0^2.
\end{equation}

Next we control the jump term. Choose $\boldsymbol{\tau}_2 \in \boldsymbol{\Sigma}_{h}$ such that all the degrees of freedom (cf. Lemma~2.1 in \cite{ChenHuHuang2015}) for $\boldsymbol{\tau}_2$ vanish except the following one:
\[
\int_{F}(\boldsymbol{\tau}_2\boldsymbol{\nu}_F)\cdot \boldsymbol{w}\,\dd s = h_F^{-1}\int_F [\boldsymbol{v}_h]\cdot \boldsymbol{w}\,\dd s \quad \forall~\boldsymbol{w}\in\boldsymbol{P}_{1}(F; \mathbb{R}^n) \textrm{ on each face } F.
\]
Then we have
\begin{align*}
\int_{K}\boldsymbol{\tau}_2:\boldsymbol{\varepsilon}_h(\boldsymbol{v}_h)\,\dd x = 0, \quad
\int_{F}(\boldsymbol{\tau}_2\boldsymbol{\nu}_F)\cdot \boldsymbol{\pi}_F[\boldsymbol{v}_h]\,\dd s  = h_F^{-1}\|\boldsymbol{\pi}_F[\boldsymbol{v}_h]\|_{0,F}^2,
\end{align*}
\begin{equation}\label{eq:temp1}
\|\boldsymbol{\tau}_2\|_0^2\lesssim \sum_{F\in\mathcal{T}_h}h_F^{-1}\|\boldsymbol{\pi}_F[\boldsymbol{v}_h]\|_{0,F}^2.
\end{equation}
Thus by \eqref{eq:piFestimate} and \eqref{eq:temp1}, there exists a constant $C_2>0$ such that
\begin{align}
b(\boldsymbol{\tau}_2, \boldsymbol{v}_h) & = \sum_{F\in\mathcal{F}_h} \int_{F}(\boldsymbol{\tau}_2\boldsymbol{\nu}_F)\cdot [\boldsymbol{v}_h]\,\dd s \notag\\
& = \sum_{F\in\mathcal{F}_h} \int_{F}(\boldsymbol{\tau}_2\boldsymbol{\nu}_F)\cdot ([\boldsymbol{v}_h]-\boldsymbol{\pi}_F[\boldsymbol{v}_h])\,\dd s + \sum_{F\in\mathcal{F}_h}h_F^{-1}\|\boldsymbol{\pi}_F[\boldsymbol{v}_h]\|_{0,F}^2 \notag\\
& \geq  -C_2\|\boldsymbol{\varepsilon}_h(\boldsymbol{v}_h)\|_0^2+\frac{1}{2}\sum_{F\in\mathcal{F}_h}h_F^{-1}\|\boldsymbol{\pi}_F[\boldsymbol{v}_h]\|_{0,F}^2. \label{eq:temp3}
\end{align}

Now taking $\boldsymbol{\tau}_h=\boldsymbol{\tau}_1+\frac{C_1}{2C_2}\boldsymbol{\tau}_2$, it holds from
\eqref{eq:temp2} and \eqref{eq:temp3}
\begin{align*}
b(\boldsymbol{\tau}_h, \boldsymbol{v}_h)&=b(\boldsymbol{\tau}_1, \boldsymbol{v}_h) + \frac{C_1}{2C_2} b(\boldsymbol{\tau}_2, \boldsymbol{v}_h) \\
&\geq  \frac{C_1}{2}\|\boldsymbol{\varepsilon}_h(\boldsymbol{v}_h)\|_0^2+\frac{C_1}{4C_2} \sum_{F\in\mathcal{F}_h}h_F^{-1}\|\boldsymbol{\pi}_F[\boldsymbol{v}_h]\|_{0,F}^2.
\end{align*}
Thanks to \eqref{eq:discreteequivnorm}, we get
\[
|\boldsymbol{v}_h|_{1,h}^2\lesssim b(\boldsymbol{\tau}_h, \boldsymbol{v}_h).
\]
On the other hand, it follows from the inverse inequality, \eqref{eq:temp4} and \eqref{eq:temp1}
\[
\|\boldsymbol{\tau}_h\|_{0,h}\lesssim \|\boldsymbol{\tau}_h\|_{0}\lesssim |\boldsymbol{v}_h|_{1,h}.
\]
Finally the inf-sup condition \eqref{eq:infsup2} is the result of the last two inequalities and  consequently \eqref{eq:infsup22} holds by taking $\hat{\boldsymbol{\tau}}_h=\boldsymbol{\tau}_h-(\frac{1}{n}\int_{\Omega}\textrm{tr}\boldsymbol{\tau}_h\,\dd x) \boldsymbol{\delta}$.
\end{proof}

\subsection{Coercivity in the null space of the div operator}
Besides the inf-sup condition, another issue of the linear elasticity in the mixed form is the coercivity of bilinear form $a(\cdot,\cdot)$. On the whole space: for all $\bs \sigma \in \bs \Sigma$,
\begin{equation}\label{cor1}
a(\bs \sigma, \bs \sigma)\geq \frac{1}{n\lambda + 2\mu}\|\bs \sigma\|_0^2.
\end{equation}
The coercivity constant, unfortunately, is in the order of $\mcal O(1/\lambda)$ as $\lambda \to +\infty$. Namely it is not robust to $\lambda$. To obtain a robust coercivity, we first recall the following inequality which implies the coercivity in the null space of the div operator.

\begin{lemma}[Proposition 9.1.1 in \cite{BoffiBrezziFortin2013}]\label{lem:temp1}
For $\boldsymbol{\tau}\in\boldsymbol{H}(\mathbf{div}, \Omega; \mathbb{S})$ satisfying $\int_{\Omega}\textrm{tr}\boldsymbol{\tau}\,\dd x=0$, we have
\[
\|\boldsymbol{\tau}\|_0\lesssim \|\boldsymbol{\tau}\|_a + \|\div \bs \tau\|_{-1},
\]
where $\|\boldsymbol{\tau}\|_a^2:=a(\boldsymbol{\tau}, \boldsymbol{\tau})$ and $\|\div \bs \tau\|_{-1} = \sup_{\boldsymbol{v}\in \boldsymbol{H}_0^1(\Omega; \mathbb{R}^n)}b(\boldsymbol{\tau}, \boldsymbol{v})/|\boldsymbol{v}|_1$.
\end{lemma}

We then move to the discrete case. Define discrete norms
\begin{align*}
\|\div \bs \tau\|_{-1,h}  &:=  \sup_{\boldsymbol{v}_h\in \boldsymbol{V}_{h}}\frac{b(\boldsymbol{\tau}, \boldsymbol{v}_h)}{|\boldsymbol{v}_h|_{1,h}},\\
\|h\div \bs \tau\|^2  &:=  \sum_{K\in\mathcal{T}_h}h_K^2\|\mathbf{div}\boldsymbol{\tau}\|_{0,K}^2.
\end{align*}
Let $\boldsymbol{Q}_h^{k-1}$ be the $L^2$ orthogonal projection from $\boldsymbol{L}^2(\Omega; \mathbb{R}^n)$ onto $\boldsymbol{V}_{h}$, which will be abbreviated as $\boldsymbol{Q}_h$. It holds the following error estimate (cf. \cite{Ciarlet1978, BrennerScott2008})
\begin{equation}\label{eq:L2projErrorEstimate}
\|\boldsymbol{v}-\boldsymbol{Q}_h\boldsymbol{v}\|_{0,K}+h_K^{1/2}\|\boldsymbol{v}-\boldsymbol{Q}_h\boldsymbol{v}\|_{0,\partial K}\lesssim h_K^{\min\{k,m\}}|\boldsymbol{v}|_{m,K} \quad \forall~\boldsymbol{v}\in \boldsymbol{H}^{m}(\Omega; \mathbb{R}^n)
\end{equation}
with integer $m\geq1$.

\begin{lemma}\label{lem:temp2}
For any $\boldsymbol{\tau}\in\boldsymbol{H}(\mathbf{div}, \Omega; \mathbb{S})$ satisfying $\int_{\Omega}\textrm{tr}\boldsymbol{\tau}\,\dd x=0$, we have
\[
\|\boldsymbol{\tau}\|_0\lesssim \|\boldsymbol{\tau}\|_a + \|h\div \bs \tau\| + \| \div \boldsymbol{\tau}\|_{-1,h}.
\]
\end{lemma}
\begin{proof}
It is sufficient to prove the case $k=1$. Let $\boldsymbol{v}\in \boldsymbol{H}_0^1(\Omega; \mathbb{R}^n)$, then it follows from the Cauchy-Schwarz inequality and \eqref{eq:L2projErrorEstimate}
\begin{align*}
b(\boldsymbol{\tau}, \boldsymbol{v})&=b(\boldsymbol{\tau}, \boldsymbol{v}-\boldsymbol{Q}_h\boldsymbol{v})+b(\boldsymbol{\tau}, \boldsymbol{Q}_h\boldsymbol{v}) \\
& \lesssim \|h\div \bs \tau\||\boldsymbol{v}|_1+b(\boldsymbol{\tau}, \boldsymbol{Q}_h\boldsymbol{v}).
\end{align*}
Again by \eqref{eq:L2projErrorEstimate}, it holds
\begin{equation}\label{eq:temp9}
|\boldsymbol{Q}_h\boldsymbol{v}|_{1,h}^2=\sum_{F\in \mathcal F_h} h_F^{-1}\|[\boldsymbol{Q}_h\boldsymbol{v}]\|_{0,F}^2=\sum_{F\in \mathcal F_h} h_F^{-1}\|[\boldsymbol{Q}_h\boldsymbol{v}-\boldsymbol{v}]\|_{0,F}^2\lesssim |\boldsymbol{v}|_1^2.
\end{equation}
Hence we get from the last two inequalities
\[
\|\div \bs \tau\|_{-1} = \sup_{\boldsymbol{v}\in \boldsymbol{H}_0^1(\Omega; \mathbb{R}^n)}\frac{b(\boldsymbol{\tau}, \boldsymbol{v})}{|\boldsymbol{v}|_1}\lesssim \|h\div \bs \tau\| + \sup_{\boldsymbol{v}_h\in \boldsymbol{V}_{h}}\frac{b(\boldsymbol{\tau}, \boldsymbol{v}_h)}{|\boldsymbol{v}_h|_{1,h}}.
\]
Therefore we can end the proof by using Lemma~\ref{lem:temp1}.
\end{proof}

\subsection{Stability in mesh dependent norms}
We now present stability in mesh dependent norms. For $k\geq n+1$, since there is no stabilization term and $\div \bs \Sigma_h \subset \bs V_h$, then $\ker(\div)\cap \bs \Sigma_h \subset \ker(\div)\cap \bs \Sigma$. The stability follows from  Lemma~\ref{lem:temp1} and inf-sup condition \eqref{eq:infsup22}.
\begin{theorem}
For $k\geq n+1$, it follows for any $(\widetilde{\boldsymbol{\sigma}}_h, \widetilde{\boldsymbol{u}}_h)\in \hat{\boldsymbol{\Sigma}}_{h}\times \boldsymbol{V}_{h}$,
\begin{equation}\label{eq:infsup3}
\|\widetilde{\boldsymbol{\sigma}}_h\|_{0,h} + |\widetilde{\boldsymbol{u}}_h|_{1,h}\lesssim \sup_{(\boldsymbol{\tau}_h, \boldsymbol{v}_h)\in \hat{\boldsymbol{\Sigma}}_{h}\times \boldsymbol{V}_{h}}\frac{\mathbb A(\widetilde{\boldsymbol{\sigma}}_h, \widetilde{\boldsymbol{u}}_h; \boldsymbol{\tau}_h, \boldsymbol{v}_h)}{\|\boldsymbol{\tau}_h\|_{0,h} + |\boldsymbol{v}_h|_{1,h}}.
\end{equation}
\end{theorem}

\begin{corollary}
Let $k\geq n+1$.
Assume that $\boldsymbol{\sigma}\in\boldsymbol{H}^{k+1}(\Omega; \mathbb{S})$ and $\boldsymbol{u}\in\boldsymbol{H}^{k}(\Omega; \mathbb{R}^n)$, then
\begin{equation}\label{eq:errorestimate1}
\|\boldsymbol{\sigma}-\boldsymbol{\sigma}_h\|_{0,h} + |\boldsymbol{Q}_h\boldsymbol{u}-\boldsymbol{u}_h|_{1,h}
\lesssim h^{k+1}\|\boldsymbol{\sigma}\|_{k+1},
\end{equation}
\begin{equation}\label{eq:errorestimate2}
|\boldsymbol{u}-\boldsymbol{u}_h|_{1,h}
\lesssim h^{k-1}\left (\|\boldsymbol{\sigma}\|_{k+1}+\|\boldsymbol{u}\|_{k}\right ).
\end{equation}
Moreover, when $\Omega$ is convex, we have
\begin{equation}\label{eq:errorestimate3}
\|\boldsymbol{Q}_h\boldsymbol{u}-\boldsymbol{u}_h\|_{0}
\lesssim h^{k+2}\|\boldsymbol{\sigma}\|_{k+1}.
\end{equation}
\end{corollary}
\begin{proof}
Subtracting \eqref{stabmfem1}-\eqref{stabmfem2} from \eqref{mixedform1}-\eqref{mixedform2}, we get the error equation
\begin{align}
a(\boldsymbol{\sigma}-\boldsymbol{\sigma}_h,\boldsymbol{\tau}_h)+b(\boldsymbol{\tau}_h, \boldsymbol{u}-\boldsymbol{u}_h)&=0 \quad\quad \forall\,\boldsymbol{\tau}_h\in\boldsymbol{\Sigma}_{h}, \label{erroreqn1} \\
b(\boldsymbol{\sigma}-\boldsymbol{\sigma}_h, \boldsymbol{v}_h)&=0 \quad\quad \forall\,\boldsymbol{v}_h\in \boldsymbol{V}_{h}. \label{erroreqn2}
\end{align}
Let $\boldsymbol{I}_h^{HZ}$ be the standard interpolation from $\boldsymbol{H}^1(\Omega; \mathbb{S})$ to $\boldsymbol{\Sigma}_{h}$ defined in \cite[Remark 3.1]{Hu2015a}, and it holds
\begin{equation}\label{eq:opcommutative}
\mathbf{div}(\boldsymbol{I}_h^{HZ}\boldsymbol{\sigma})=\boldsymbol{Q}_h(\mathbf{div}\boldsymbol{\sigma}).
\end{equation}
Thus we have from \eqref{erroreqn2}
\[
b(\boldsymbol{I}_h^{HZ}\boldsymbol{\sigma}-\boldsymbol{\sigma}_h, \boldsymbol{v}_h)=b(\boldsymbol{\sigma}-\boldsymbol{\sigma}_h, \boldsymbol{v}_h)=0.
\]
By the definition of $\boldsymbol{Q}_h$ and  \eqref{erroreqn1},
\[
b(\boldsymbol{\tau}_h, \boldsymbol{Q}_h\boldsymbol{u}-\boldsymbol{u}_h)=b(\boldsymbol{\tau}_h, \boldsymbol{u}-\boldsymbol{u}_h)=-a(\boldsymbol{\sigma}-\boldsymbol{\sigma}_h,\boldsymbol{\tau}_h).
\]
Combining the last two equalities, it holds
\begin{align*}
&\mathbb A(\boldsymbol{I}_h^{HZ}\boldsymbol{\sigma}-\boldsymbol{\sigma}_h, \boldsymbol{Q}_h\boldsymbol{u}-\boldsymbol{u}_h; \boldsymbol{\tau}_h, \boldsymbol{v}_h) \\
=&a(\boldsymbol{I}_h^{HZ}\boldsymbol{\sigma}-\boldsymbol{\sigma}_h,\boldsymbol{\tau}_h)+b(\boldsymbol{\tau}_h, \boldsymbol{Q}_h\boldsymbol{u}-\boldsymbol{u}_h)+b(\boldsymbol{I}_h^{HZ}\boldsymbol{\sigma}-\boldsymbol{\sigma}_h, \boldsymbol{v}_h) \\
=&a(\boldsymbol{I}_h^{HZ}\boldsymbol{\sigma}-\boldsymbol{\sigma},\boldsymbol{\tau}_h),
\end{align*}
which together with \eqref{eq:infsup3} implies
\[
\|\boldsymbol{I}_h^{HZ}\boldsymbol{\sigma}-\boldsymbol{\sigma}_h\|_{0,h} + |\boldsymbol{Q}_h\boldsymbol{u}-\boldsymbol{u}_h|_{1,h}\lesssim \|\boldsymbol{I}_h^{HZ}\boldsymbol{\sigma}-\boldsymbol{\sigma}\|_{0,h}.
\]
Therefore we will achieve \eqref{eq:errorestimate1}-\eqref{eq:errorestimate2} by using the last inequality, and the error estimate of $\boldsymbol{I}_h^{HZ}$ and $\boldsymbol{Q}_h$. The error estimate \eqref{eq:errorestimate3} can be derived by using the duality argument as in \cite{DouglasRoberts1985, Stenberg1991}.
\end{proof}
\begin{remark}\rm
The optimal convergence rate of $\|\boldsymbol{\sigma}-\boldsymbol{\sigma}_h\|_{0,h}$ has been mentioned in \cite[Remarks~3.1-3.2]{Hu2015a} and \cite[Remarks~3.6]{HuZhang2015}, but the $2$-order higher superconvergent rates of $|\boldsymbol{Q}_h\boldsymbol{u}-\boldsymbol{u}_h|_{1,h}$ and $\|\boldsymbol{Q}_h\boldsymbol{u}-\boldsymbol{u}_h\|_{0}$ are new which can be used to reconstruct a better approximation of displacement. The convergence rate of $|\boldsymbol{u}-\boldsymbol{u}_h|_{1,h}$ is also optimal.
\end{remark}

Due to the stabilization term (for inf-sup condition), our proof of the stability is more complicated for the low order case $1\leq k\leq n$.
\begin{theorem}
For $1\leq k\leq n$, it holds for any $(\widetilde{\boldsymbol{\sigma}}_h, \widetilde{\boldsymbol{u}}_h)\in \hat{\boldsymbol{\Sigma}}_{h}\times \boldsymbol{V}_{h}$ that
\begin{equation}\label{eq:infsup5}
\|\widetilde{\boldsymbol{\sigma}}_h\|_{0,h} + |\widetilde{\boldsymbol{u}}_h|_{1,h}\lesssim \sup_{(\boldsymbol{\tau}_h, \boldsymbol{v}_h)\in \hat{\boldsymbol{\Sigma}}_{h}\times \boldsymbol{V}_{h}}\frac{\mathbb A(\widetilde{\boldsymbol{\sigma}}_h, \widetilde{\boldsymbol{u}}_h; \boldsymbol{\tau}_h, \boldsymbol{v}_h)}{\|\boldsymbol{\tau}_h\|_{0,h}+ |\boldsymbol{v}_h|_{1,h}}.
\end{equation}
\end{theorem}
\begin{proof}
As demonstrated in Lemma~\ref{lem:infsup22}, it is equivalent to prove
\begin{equation}\label{eq:infsup52}
\|\widetilde{\boldsymbol{\sigma}}_h\|_{0,h} + |\widetilde{\boldsymbol{u}}_h|_{1,h}\lesssim \sup_{(\boldsymbol{\tau}_h, \boldsymbol{v}_h)\in \boldsymbol{\Sigma}_{h}\times \boldsymbol{V}_{h}}\frac{\mathbb A(\widetilde{\boldsymbol{\sigma}}_h, \widetilde{\boldsymbol{u}}_h; \boldsymbol{\tau}_h, \boldsymbol{v}_h)}{\|\boldsymbol{\tau}_h\|_{0,h}+|\boldsymbol{v}_h|_{1,h}}:=\beta.
\end{equation}
The notation $\beta$ is introduced just for ease of presentation.
Let $\boldsymbol{\tau}_1=\boldsymbol{E}\boldsymbol{\varepsilon}_h(\widetilde{\boldsymbol{u}}_h)$ for $k\geq 2$ and $\boldsymbol{\tau}_1=\boldsymbol{0}$ for $k=1$, then it holds from Cauchy-Schwarz inequality
\[
\mathbb A(\widetilde{\boldsymbol{\sigma}}_h, \widetilde{\boldsymbol{u}}_h; \boldsymbol{\tau}_1, 0)=a(\widetilde{\boldsymbol{\sigma}}_h, \boldsymbol{\tau}_1)+b(\boldsymbol{\tau}_1, \widetilde{\boldsymbol{u}}_h)\geq -\|\widetilde{\boldsymbol{\sigma}}_h\|_a\|\boldsymbol{\tau}_1\|_a + b(\boldsymbol{\tau}_1, \widetilde{\boldsymbol{u}}_h).
\]
Using \eqref{eq:temp4}-\eqref{eq:temp2}, there exists a constant $C_3 > 0$ such that
\begin{align}
\mathbb A(\widetilde{\boldsymbol{\sigma}}_h, \widetilde{\boldsymbol{u}}_h; \boldsymbol{\tau}_1, 0)& \geq  C_1\|\boldsymbol{\varepsilon}_h(\widetilde{\boldsymbol{u}}_h)\|_0^2-C_3\|\widetilde{\boldsymbol{\sigma}}_h\|_a\|\boldsymbol{\varepsilon}_h(\widetilde{\boldsymbol{u}}_h)\|_0 \notag\\
& \geq  \frac{C_1}{2}\|\boldsymbol{\varepsilon}_h(\widetilde{\boldsymbol{u}}_h)\|_0^2-\frac{C_3^2}{2C_1}\|\widetilde{\boldsymbol{\sigma}}_h\|_a^2. \label{eq:temp5}
\end{align}
Let $\boldsymbol{v}_1\in\boldsymbol{V}_{h}$ such that $\boldsymbol{v}_1|_K=h_K^2\mathbf{div}\widetilde{\boldsymbol{\sigma}}_h$ for each $K\in\mathcal{T}_h$. Applying inverse inequality, we have
\begin{equation}\label{eq:temp7}
|\boldsymbol{v}_1|_{1,h}\lesssim \| h\div \widetilde{\boldsymbol{\sigma}}_h\|\lesssim \|\widetilde{\boldsymbol{\sigma}}_h\|_0.
\end{equation}
Thus there exists a constant $C_4 > 0$ such that
\begin{align}
\mathbb A(\widetilde{\boldsymbol{\sigma}}_h, \widetilde{\boldsymbol{u}}_h; 0, \boldsymbol{v}_1)& =  b(\widetilde{\boldsymbol{\sigma}}_h, \boldsymbol{v}_1)-c(\widetilde{\boldsymbol{u}}_h,\boldsymbol{v}_1)\geq\|h\div \widetilde{\boldsymbol{\sigma}}_h\|^2-\|\widetilde{\boldsymbol{u}}_h\|_c\|\boldsymbol{v}_1\|_c \notag\\
& \geq  \|h\div \widetilde{\boldsymbol{\sigma}}_h\|^2 - C_4\|\widetilde{\boldsymbol{u}}_h\|_c\|h\div \widetilde{\boldsymbol{\sigma}}_h\| \notag\\
& \geq  \frac{1}{2}\|h\div \widetilde{\boldsymbol{\sigma}}_h\|^2 - \frac{C_4^2}{2}\|\widetilde{\boldsymbol{u}}_h\|_c^2. \label{eq:temp6}
\end{align}
Now taking $\boldsymbol{\tau}_h=\widetilde{\boldsymbol{\sigma}}_h+\frac{C_1}{C_3^2}\boldsymbol{\tau}_1$ and $\boldsymbol{v}_h=-\widetilde{\boldsymbol{u}}_h+\frac{1}{C_4^2}\boldsymbol{v}_1$, we have from \eqref{eq:temp5} and \eqref{eq:temp6}
\begin{align*}
&\mathbb A(\widetilde{\boldsymbol{\sigma}}_h, \widetilde{\boldsymbol{u}}_h; \boldsymbol{\tau}_h, \boldsymbol{v}_h)\\
& =\mathbb A(\widetilde{\boldsymbol{\sigma}}_h, \widetilde{\boldsymbol{u}}_h; \widetilde{\boldsymbol{\sigma}}_h, -\widetilde{\boldsymbol{u}}_h)+\frac{C_1}{C_3^2}\mathbb A(\widetilde{\boldsymbol{\sigma}}_h, \widetilde{\boldsymbol{u}}_h; \boldsymbol{\tau}_1, 0)+\frac{1}{C_4^2}\mathbb A(\widetilde{\boldsymbol{\sigma}}_h, \widetilde{\boldsymbol{u}}_h; 0, \boldsymbol{v}_1) \\
 & \geq \frac{1}{2}\|\widetilde{\boldsymbol{\sigma}}_h\|_{a}^2 +\frac{1}{2}\|\widetilde{\boldsymbol{u}}_h\|_c^2 + \frac{1}{2C_4^2}\|h\div \widetilde{\boldsymbol{\sigma}}_h\|^2  + \frac{C_1^2}{2C_3^2}\|\boldsymbol{\varepsilon}_h(\widetilde{\boldsymbol{u}}_h)\|_0^2,
\end{align*}
which together with \eqref{eq:temp4} and \eqref{eq:temp7} indicates
\[
\|\widetilde{\boldsymbol{\sigma}}_h\|_{a}^2 + \|h\div \widetilde{\boldsymbol{\sigma}}_h\|^2 + |\widetilde{\boldsymbol{u}}_h|_{1,h}^2\lesssim \mathbb A(\widetilde{\boldsymbol{\sigma}}_h, \widetilde{\boldsymbol{u}}_h; \boldsymbol{\tau}_h, \boldsymbol{v}_h)\lesssim \beta(\|\widetilde{\boldsymbol{\sigma}}_h\|_{0}+|\widetilde{\boldsymbol{u}}_h|_{1,h}).
\]
According to Lemma~\ref{lem:temp2} and the definition of $\beta$, it holds
\begin{align*}
\|\widetilde{\boldsymbol{\sigma}}_h\|_0& \lesssim \|\widetilde{\boldsymbol{\sigma}}_h\|_a + \|h\div \widetilde{\boldsymbol{\sigma}}_h\| + \sup_{\boldsymbol{v}_h\in \boldsymbol{V}_{h}}\frac{b(\widetilde{\boldsymbol{\sigma}}_h, \boldsymbol{v}_h)}{|\boldsymbol{v}_h|_{1,h}} \\
& = \|\widetilde{\boldsymbol{\sigma}}_h\|_a + \|h\div \widetilde{\boldsymbol{\sigma}}_h\| + \sup_{\boldsymbol{v}_h\in \boldsymbol{V}_{h}}\frac{\mathbb A(\widetilde{\boldsymbol{\sigma}}_h, \widetilde{\boldsymbol{u}}_h; \boldsymbol{0}, \boldsymbol{v}_h)+ c(\widetilde{\boldsymbol{u}}_h, \boldsymbol{v}_h)}{|\boldsymbol{v}_h|_{1,h}} \\
& \lesssim  \|\widetilde{\boldsymbol{\sigma}}_h\|_a + \|h\div \widetilde{\boldsymbol{\sigma}}_h\| + |\widetilde{\boldsymbol{u}}_h|_{1,h}+\beta.
\end{align*}
Thus we obtain from the last two inequalities
\begin{align*}
\|\widetilde{\boldsymbol{\sigma}}_h\|_0^2+|\widetilde{\boldsymbol{u}}_h|_{1,h}^2& \lesssim \|\widetilde{\boldsymbol{\sigma}}_h\|_a^2 + \|h\div \widetilde{\boldsymbol{\sigma}}_h\|^2 + |\widetilde{\boldsymbol{u}}_h|_{1,h}^2+\beta^2 \\
& \lesssim  \beta(\|\widetilde{\boldsymbol{\sigma}}_h\|_{0}+|\widetilde{\boldsymbol{u}}_h|_{1,h})+\beta^2,
\end{align*}
which implies inf-sup condition \eqref{eq:infsup52}.
\end{proof}

\begin{corollary}
Let $1\leq k\leq n$.
Assume that $\boldsymbol{\sigma}\in\boldsymbol{H}^{k+1}(\Omega; \mathbb{S})$ and $\boldsymbol{u}\in\boldsymbol{H}^{k}(\Omega; \mathbb{R}^n)$, then
\[
\|\boldsymbol{\sigma}-\boldsymbol{\sigma}_h\|_{0,h} + |\boldsymbol{u}-\boldsymbol{u}_h|_{1,h}
\lesssim h^{k-1}\left (\|\boldsymbol{\sigma}\|_{k+1}+\|\boldsymbol{u}\|_{k}\right ).
\]
\end{corollary}
The convergence of rate of $|\boldsymbol{u}-\boldsymbol{u}_h|_{1,h}$ is optimal. But the $L^2$-type error of $\|\boldsymbol{\sigma}-\boldsymbol{\sigma}_h\|_{0,h}$ is two order less.

\begin{remark}\rm
Using the stability in mesh dependent norms established in \cite{LovadinaStenberg2006, BraessVerfurth1996}, the MINRES method with additive Schwarz preconditioner was developed for the mixed finite element methods of the Poisson problem in \cite{RustenVassilevskiWinther1996}, and the CG method with auxiliary space preconditioner for the corresponding Schur complement problem was designed in \cite{Hannukainen2012}.
Similar stability in mesh dependent norm for the mixed finite macroelement methods of the linear elasticity can be found in \cite{Stenberg1986}, hence the fast auxiliary space preconditioner constructed in this paper can be easily extended to these mixed methods. For example, the block-triangular preconditioner similar to \eqref{eq:Matrix_Pre} for the mixed Poisson problem has been included in $i$FEM \cite{Chen.L2008c} since 2012.
\end{remark}

\subsection{Postprocessing}

Based the superconvergent results of the displacement in \eqref{eq:errorestimate1} and \eqref{eq:errorestimate3}, we will construct a superconvergent postprocessed displacement from $(\boldsymbol{\sigma}_h, \boldsymbol{u}_h)$ for the higher order case $k\geq n+1$ in this subsection.

To this end, let
\[
\boldsymbol{V}_{h}^{\ast}:=\left\{\boldsymbol{v}\in \boldsymbol{L}^2(\Omega; \mathbb{R}^n): \boldsymbol{v}|_K\in \boldsymbol{P}_{k+1}(K; \mathbb{R}^n)\quad \forall\,K\in\mathcal
{T}_h\right\}.
\]
Then a postprocessed displacement can be defined as follows:
Find $\boldsymbol{u}_h^{\ast}\in\boldsymbol{V}_{h}^{\ast}$ such that
\begin{equation}\label{postprocess1}
\boldsymbol{Q}_h\boldsymbol{u}_h^{\ast}=\boldsymbol{u}_h,
\end{equation}
\begin{equation}\label{postprocess2}
(\boldsymbol{\varepsilon}(\boldsymbol{u}_h^{\ast}), \boldsymbol{\varepsilon}(\boldsymbol{v}))_K=(\mathfrak{A}\boldsymbol{\sigma}_h, \boldsymbol{\varepsilon}(\boldsymbol{v}))_K \quad \forall~\boldsymbol{v}\in (\boldsymbol{I}-\boldsymbol{Q}_h)\boldsymbol{V}_{h}^{\ast}|_K,
\end{equation}
for any $K\in\mathcal{T}_h$.
To derive the error estimate for the postprocessed displacement $\boldsymbol{u}_h^{\ast}$, we will merge the mixed finite element method~\eqref{stabmfem1}-\eqref{stabmfem2} and the postprocessing \eqref{postprocess1}-\eqref{postprocess2} into one method as in \cite{LovadinaStenberg2006}.
To be specific,
find $(\boldsymbol{\sigma}_h, \boldsymbol{u}_h^{\ast})\in \boldsymbol{\Sigma}_{h}\times \boldsymbol{V}_{h}^{\ast}$ such that
\begin{equation}\label{postmfem}
\mathbb A_h(\boldsymbol{\sigma}_h, \boldsymbol{u}_h^{\ast}; \boldsymbol{\tau}_h, \boldsymbol{v}_h^{\ast})=-(\boldsymbol{Q}_h\boldsymbol{f}, \boldsymbol{v}_h^{\ast}) \quad \forall~(\boldsymbol{\tau}_h, \boldsymbol{v}_h^{\ast})\in \boldsymbol{\Sigma}_{h}\times \boldsymbol{V}_{h}^{\ast},
\end{equation}
where
\[
\mathbb A_h(\boldsymbol{\sigma}_h, \boldsymbol{u}_h^{\ast}; \boldsymbol{\tau}_h, \boldsymbol{v}_h^{\ast}):=\mathbb A(\boldsymbol{\sigma}_h, \boldsymbol{u}_h^{\ast}; \boldsymbol{\tau}_h, \boldsymbol{v}_h^{\ast}) + (\boldsymbol{\varepsilon}_h(\boldsymbol{u}_h^{\ast})-\mathfrak{A}\boldsymbol{\sigma}_h, \boldsymbol{\varepsilon}_h(\boldsymbol{v}_h^{\ast}-\boldsymbol{Q}_h\boldsymbol{v}_h^{\ast})).
\]

\begin{lemma}
The mixed finite element method~\eqref{stabmfem1}-\eqref{stabmfem2} and the problem \eqref{postmfem} are equivalent in the following sense: if $(\boldsymbol{\sigma}_h, \boldsymbol{u}_h^{\ast})\in \boldsymbol{\Sigma}_{h}\times \boldsymbol{V}_{h}^{\ast}$ is the solution of the problem \eqref{postmfem} and let $\boldsymbol{u}_h=\boldsymbol{Q}_h\boldsymbol{u}_h^{\ast}$, then
$(\boldsymbol{\sigma}_h, \boldsymbol{u}_h)\in \boldsymbol{\Sigma}_{h}\times \boldsymbol{V}_{h}$ solves the mixed finite element method~\eqref{stabmfem1}-\eqref{stabmfem2}; Conversely, if $(\boldsymbol{\sigma}_h, \boldsymbol{u}_h)\in \boldsymbol{\Sigma}_{h}\times \boldsymbol{V}_{h}$ is the solution of the mixed finite element method~\eqref{stabmfem1}-\eqref{stabmfem2}
and $\boldsymbol{u}_h^{\ast}\in\boldsymbol{V}_{h}^{\ast}$ is the postprocessed displacement defined by \eqref{postprocess1}-\eqref{postprocess2}, then $(\boldsymbol{\sigma}_h, \boldsymbol{u}_h^{\ast})\in \boldsymbol{\Sigma}_{h}\times \boldsymbol{V}_{h}^{\ast}$ solves the problem \eqref{postmfem}.
\end{lemma}
\begin{proof}
Taking any $(\boldsymbol{\tau}_h, \boldsymbol{v}_h)\in \boldsymbol{\Sigma}_{h}\times \boldsymbol{V}_{h}$, and noting the fact that $\boldsymbol{v}_h=\boldsymbol{Q}_h\boldsymbol{v}_h$ and $\mathbf{div}\boldsymbol{\Sigma}_{h}\subset \boldsymbol{V}_{h}$,
we have
\begin{align}
\mathbb A_h(\boldsymbol{\sigma}_h, \boldsymbol{u}_h^{\ast}; \boldsymbol{\tau}_h, \boldsymbol{v}_h)=&\mathbb A(\boldsymbol{\sigma}_h, \boldsymbol{u}_h^{\ast}; \boldsymbol{\tau}_h, \boldsymbol{v}_h)
=a(\boldsymbol{\sigma}_h,\boldsymbol{\tau}_h)+b(\boldsymbol{\tau}_h, \boldsymbol{u}_h^{\ast})+b(\boldsymbol{\sigma}_h, \boldsymbol{v}_h) \notag\\
=& a(\boldsymbol{\sigma}_h,\boldsymbol{\tau}_h)+b(\boldsymbol{\tau}_h, \boldsymbol{Q}_h\boldsymbol{u}_h^{\ast})+b(\boldsymbol{\sigma}_h, \boldsymbol{v}_h)=\mathbb A(\boldsymbol{\sigma}_h, \boldsymbol{Q}_h\boldsymbol{u}_h^{\ast}; \boldsymbol{\tau}_h, \boldsymbol{v}_h). \label{eq:temp1post}
\end{align}
Hence we can see from \eqref{eq:temp1post} that $(\boldsymbol{\sigma}_h, \boldsymbol{u}_h)$ solves the mixed finite element method~\eqref{stabmfem1}-\eqref{stabmfem2} if $(\boldsymbol{\sigma}_h, \boldsymbol{u}_h^{\ast})$ is the solution of the problem \eqref{postmfem}.

Conversely, since $\mathbf{div}\boldsymbol{\Sigma}_{h}\subset \boldsymbol{V}_{h}$ and $(\boldsymbol{I}-\boldsymbol{Q}_h)^2=\boldsymbol{I}-\boldsymbol{Q}_h$, it follows from \eqref{eq:temp1post} and \eqref{postprocess1}
\begin{align}
\mathbb A_h(\boldsymbol{\sigma}_h, \boldsymbol{u}_h^{\ast}; \boldsymbol{\tau}_h, \boldsymbol{v}_h^{\ast})=&\mathbb A_h(\boldsymbol{\sigma}_h, \boldsymbol{u}_h^{\ast}; \boldsymbol{\tau}_h, \boldsymbol{Q}_h\boldsymbol{v}_h^{\ast})+\mathbb A_h(\boldsymbol{\sigma}_h, \boldsymbol{u}_h^{\ast}; \boldsymbol{0}, \boldsymbol{v}_h^{\ast}-\boldsymbol{Q}_h\boldsymbol{v}_h^{\ast}) \notag\\
=& \mathbb A(\boldsymbol{\sigma}_h, \boldsymbol{u}_h; \boldsymbol{\tau}_h, \boldsymbol{Q}_h\boldsymbol{v}_h^{\ast})+(\boldsymbol{\varepsilon}_h(\boldsymbol{u}_h^{\ast})-\mathfrak{A}\boldsymbol{\sigma}_h, \boldsymbol{\varepsilon}_h(\boldsymbol{v}_h^{\ast}-\boldsymbol{Q}_h\boldsymbol{v}_h^{\ast})) \notag\\
=& -(\boldsymbol{f}, \boldsymbol{Q}_h\boldsymbol{v}_h^{\ast}) +(\boldsymbol{\varepsilon}_h(\boldsymbol{u}_h^{\ast})-\mathfrak{A}\boldsymbol{\sigma}_h, \boldsymbol{\varepsilon}_h(\boldsymbol{v}_h^{\ast}-\boldsymbol{Q}_h\boldsymbol{v}_h^{\ast})), \label{eq:temp3post}
\end{align}
which together with \eqref{postprocess2} means that $(\boldsymbol{\sigma}_h, \boldsymbol{u}_h^{\ast})$ solves the problem \eqref{postmfem}.
\end{proof}

\begin{lemma}
For any $\boldsymbol{v}\in \boldsymbol{H}^1(\mathcal{T}_h; \mathbb{R}^n)$, it holds that
\begin{equation}\label{eq:temp4post}
|\boldsymbol{v}-\boldsymbol{Q}_h\boldsymbol{v}|_{1,h} \eqsim \|\boldsymbol{\varepsilon}_h (\boldsymbol{v}-\boldsymbol{Q}_h\boldsymbol{v})\|_0.
\end{equation}
\end{lemma}
\begin{proof}
It is sufficient to prove
\begin{equation}\label{eq:temp2post}
\sum_{F\in \mathcal F_h} h_F^{-1}\|[\boldsymbol{v}-\boldsymbol{Q}_h\boldsymbol{v}]\|_{0,F}^2\lesssim \|\boldsymbol{\varepsilon}_h (\boldsymbol{v}-\boldsymbol{Q}_h\boldsymbol{v})\|_0^2.
\end{equation}
Let $\boldsymbol{\pi}$ be defined as in Lemma~\ref{lem:temp32} and $\boldsymbol{w}=\boldsymbol{v}-\boldsymbol{Q}_h\boldsymbol{v}$. It follows from (3.3) in \cite{Brenner2004a}
\begin{align}
\sum_{F\in\mathcal{F}_h}h_F^{-1}\|[\boldsymbol{w}-\boldsymbol{Q}_h\boldsymbol{w}]\|_{0,F}^2
=&\sum_{F\in\mathcal{F}_h}h_F^{-1}\|[(\boldsymbol{w}-\boldsymbol{\pi}\boldsymbol{w})-\boldsymbol{Q}_h(\boldsymbol{w}-\boldsymbol{\pi}\boldsymbol{w})]\|_{0,F}^2 \notag \\
\leq &\sum_{K\in\mathcal{T}_h}|\boldsymbol{w}-\boldsymbol{\pi}\boldsymbol{w}|_{1,K}^2
\lesssim \|\boldsymbol{\varepsilon}_h(\boldsymbol{w})\|_0^2. \label{eq:piFestimatepost}
\end{align}
On the other hand,
\[
\sum_{F\in\mathcal{F}_h}h_F^{-1}\|[\boldsymbol{v}-\boldsymbol{Q}_h\boldsymbol{v}]\|_{0,F}^2=\sum_{F\in\mathcal{F}_h}h_F^{-1}\|[\boldsymbol{w}-\boldsymbol{Q}_h\boldsymbol{w}]\|_{0,F}^2.
\]
Therefore \eqref{eq:temp2post} follows from \eqref{eq:piFestimatepost}.
\end{proof}

\begin{theorem}
For any $(\widetilde{\boldsymbol{\sigma}}_h, \widetilde{\boldsymbol{u}}_h^{\ast})\in \hat{\boldsymbol{\Sigma}}_{h}\times \boldsymbol{V}_{h}^{\ast}$, it follows
\begin{equation}\label{eq:infsup2post}
\|\widetilde{\boldsymbol{\sigma}}_h\|_{0,h} + |\widetilde{\boldsymbol{u}}_h^{\ast}|_{1,h}\lesssim \sup_{(\boldsymbol{\tau}_h, \boldsymbol{v}_h^{\ast})\in \hat{\boldsymbol{\Sigma}}_{h}\times \boldsymbol{V}_{h}^{\ast}}\frac{\mathbb A_h(\widetilde{\boldsymbol{\sigma}}_h, \widetilde{\boldsymbol{u}}_h^{\ast}; \boldsymbol{\tau}_h, \boldsymbol{v}_h^{\ast})}{\|\boldsymbol{\tau}_h\|_{0,h} + |\boldsymbol{v}_h^{\ast}|_{1,h}}.
\end{equation}
\end{theorem}
\begin{proof}
For any $\boldsymbol{v}_h\in \boldsymbol{V}_h$, we have from \eqref{eq:temp1post}
\[
\mathbb A_h(\widetilde{\boldsymbol{\sigma}}_h, \widetilde{\boldsymbol{u}}_h^{\ast}; \boldsymbol{\tau}_h, \boldsymbol{v}_h)=\mathbb A(\widetilde{\boldsymbol{\sigma}}_h, \boldsymbol{Q}_h\widetilde{\boldsymbol{u}}_h^{\ast}; \boldsymbol{\tau}_h, \boldsymbol{v}_h).
\]
Since $(\widetilde{\boldsymbol{\sigma}}_h, \boldsymbol{Q}_h\widetilde{\boldsymbol{u}}_h^{\ast})\in \hat{\boldsymbol{\Sigma}}_{h}\times \boldsymbol{V}_{h}$, it holds from \eqref{eq:infsup3}
\begin{align}
\|\widetilde{\boldsymbol{\sigma}}_h\|_{0,h} + |\boldsymbol{Q}_h\widetilde{\boldsymbol{u}}_h^{\ast}|_{1,h} \lesssim & \sup_{(\boldsymbol{\tau}_h, \boldsymbol{v}_h)\in \hat{\boldsymbol{\Sigma}}_{h}\times \boldsymbol{V}_{h}}\frac{\mathbb A(\widetilde{\boldsymbol{\sigma}}_h, \boldsymbol{Q}_h\widetilde{\boldsymbol{u}}_h^{\ast}; \boldsymbol{\tau}_h, \boldsymbol{v}_h)}{\|\boldsymbol{\tau}_h\|_{0,h} + |\boldsymbol{v}_h|_{1,h}} \notag\\
=&\sup_{(\boldsymbol{\tau}_h, \boldsymbol{v}_h)\in \hat{\boldsymbol{\Sigma}}_{h}\times \boldsymbol{V}_{h}}\frac{\mathbb A_h(\widetilde{\boldsymbol{\sigma}}_h, \widetilde{\boldsymbol{u}}_h^{\ast}; \boldsymbol{\tau}_h, \boldsymbol{v}_h)}{\|\boldsymbol{\tau}_h\|_{0,h} + |\boldsymbol{v}_h|_{1,h}} \notag\\
\leq & \sup_{(\boldsymbol{\tau}_h, \boldsymbol{v}_h^{\ast})\in \hat{\boldsymbol{\Sigma}}_{h}\times \boldsymbol{V}_{h}^{\ast}}\frac{\mathbb A_h(\widetilde{\boldsymbol{\sigma}}_h, \widetilde{\boldsymbol{u}}_h^{\ast}; \boldsymbol{\tau}_h, \boldsymbol{v}_h^{\ast})}{\|\boldsymbol{\tau}_h\|_{0,h} + |\boldsymbol{v}_h^{\ast}|_{1,h}}. \label{eq:temp5post}
\end{align}
Similarly as in \eqref{eq:temp3post}, we get
\[
\mathbb A_h(\widetilde{\boldsymbol{\sigma}}_h, \widetilde{\boldsymbol{u}}_h^{\ast}; \boldsymbol{0}, \widetilde{\boldsymbol{u}}_h^{\ast}-\boldsymbol{Q}_h\widetilde{\boldsymbol{u}}_h^{\ast})=(\boldsymbol{\varepsilon}_h(\widetilde{\boldsymbol{u}}_h^{\ast})-\mathfrak{A}\widetilde{\boldsymbol{\sigma}}_h, \boldsymbol{\varepsilon}_h(\widetilde{\boldsymbol{u}}_h^{\ast}-\boldsymbol{Q}_h\widetilde{\boldsymbol{u}}_h^{\ast})).
\]
Then we rewrite it as
\begin{align}
\|\boldsymbol{\varepsilon}_h(\widetilde{\boldsymbol{u}}_h^{\ast}-\boldsymbol{Q}_h\widetilde{\boldsymbol{u}}_h^{\ast})\|_0^2=&(\mathfrak{A}\widetilde{\boldsymbol{\sigma}}_h-\boldsymbol{\varepsilon}_h(\boldsymbol{Q}_h\widetilde{\boldsymbol{u}}_h^{\ast}), \boldsymbol{\varepsilon}_h(\widetilde{\boldsymbol{u}}_h^{\ast}-\boldsymbol{Q}_h\widetilde{\boldsymbol{u}}_h^{\ast})) \notag\\
&+\mathbb A_h(\widetilde{\boldsymbol{\sigma}}_h, \widetilde{\boldsymbol{u}}_h^{\ast}; \boldsymbol{0}, \widetilde{\boldsymbol{u}}_h^{\ast}-\boldsymbol{Q}_h\widetilde{\boldsymbol{u}}_h^{\ast}). \label{eq:temp6post}
\end{align}
According to the triangle inequality and \eqref{eq:temp5post}, it holds
\begin{align*}
\|\mathfrak{A}\widetilde{\boldsymbol{\sigma}}_h-\boldsymbol{\varepsilon}_h(\boldsymbol{Q}_h\widetilde{\boldsymbol{u}}_h^{\ast})\|_0 \leq & \|\mathfrak{A}\widetilde{\boldsymbol{\sigma}}_h\|_0+\|\boldsymbol{\varepsilon}_h(\boldsymbol{Q}_h\widetilde{\boldsymbol{u}}_h^{\ast})\|_0
\lesssim \|\widetilde{\boldsymbol{\sigma}}_h\|_0 + |\boldsymbol{Q}_h\widetilde{\boldsymbol{u}}_h^{\ast}|_{1, h} \\
\lesssim & \sup_{(\boldsymbol{\tau}_h, \boldsymbol{v}_h^{\ast})\in \hat{\boldsymbol{\Sigma}}_{h}\times \boldsymbol{V}_{h}^{\ast}}\frac{\mathbb A_h(\widetilde{\boldsymbol{\sigma}}_h, \widetilde{\boldsymbol{u}}_h^{\ast}; \boldsymbol{\tau}_h, \boldsymbol{v}_h^{\ast})}{\|\boldsymbol{\tau}_h\|_{0,h} + |\boldsymbol{v}_h^{\ast}|_{1,h}}.
\end{align*}
Due to \eqref{eq:temp4post}, we have
\begin{align*}
&\mathbb A_h(\widetilde{\boldsymbol{\sigma}}_h, \widetilde{\boldsymbol{u}}_h^{\ast}; \boldsymbol{0}, \widetilde{\boldsymbol{u}}_h^{\ast}-\boldsymbol{Q}_h\widetilde{\boldsymbol{u}}_h^{\ast}) \\
\leq & \|\boldsymbol{\varepsilon}_h(\widetilde{\boldsymbol{u}}_h^{\ast}-\boldsymbol{Q}_h\widetilde{\boldsymbol{u}}_h^{\ast})\|_0 \sup_{(\boldsymbol{\tau}_h, \boldsymbol{v}_h^{\ast})\in \hat{\boldsymbol{\Sigma}}_{h}\times \boldsymbol{V}_{h}^{\ast}}\frac{\mathbb A_h(\widetilde{\boldsymbol{\sigma}}_h, \widetilde{\boldsymbol{u}}_h^{\ast}; \boldsymbol{0}, \boldsymbol{v}_h^{\ast}-\boldsymbol{Q}_h\boldsymbol{v}_h^{\ast})}{\|\boldsymbol{\varepsilon}_h(\boldsymbol{v}_h^{\ast}-\boldsymbol{Q}_h\boldsymbol{v}_h^{\ast})\|_0} \\
\lesssim &\|\boldsymbol{\varepsilon}_h(\widetilde{\boldsymbol{u}}_h^{\ast}-\boldsymbol{Q}_h\widetilde{\boldsymbol{u}}_h^{\ast})\|_0 \sup_{(\boldsymbol{\tau}_h, \boldsymbol{v}_h^{\ast})\in \hat{\boldsymbol{\Sigma}}_{h}\times \boldsymbol{V}_{h}^{\ast}}\frac{\mathbb A_h(\widetilde{\boldsymbol{\sigma}}_h, \widetilde{\boldsymbol{u}}_h^{\ast}; \boldsymbol{0}, \boldsymbol{v}_h^{\ast}-\boldsymbol{Q}_h\boldsymbol{v}_h^{\ast})}{|\boldsymbol{v}_h^{\ast}-\boldsymbol{Q}_h\boldsymbol{v}_h^{\ast}|_{1,h}} \\
\lesssim & \|\boldsymbol{\varepsilon}_h(\widetilde{\boldsymbol{u}}_h^{\ast}-\boldsymbol{Q}_h\widetilde{\boldsymbol{u}}_h^{\ast})\|_0\sup_{(\boldsymbol{\tau}_h, \boldsymbol{v}_h^{\ast})\in \hat{\boldsymbol{\Sigma}}_{h}\times \boldsymbol{V}_{h}^{\ast}}\frac{\mathbb A_h(\widetilde{\boldsymbol{\sigma}}_h, \widetilde{\boldsymbol{u}}_h^{\ast}; \boldsymbol{\tau}_h, \boldsymbol{v}_h^{\ast})}{\|\boldsymbol{\tau}_h\|_{0,h} + |\boldsymbol{v}_h^{\ast}|_{1,h}}.
\end{align*}
Using the last two inequalities and Cauchy-Schwarz inequality, we get from \eqref{eq:temp6post}
\[
\|\boldsymbol{\varepsilon}_h(\widetilde{\boldsymbol{u}}_h^{\ast}-\boldsymbol{Q}_h\widetilde{\boldsymbol{u}}_h^{\ast})\|_0\lesssim \sup_{(\boldsymbol{\tau}_h, \boldsymbol{v}_h^{\ast})\in \hat{\boldsymbol{\Sigma}}_{h}\times \boldsymbol{V}_{h}^{\ast}}\frac{\mathbb A_h(\widetilde{\boldsymbol{\sigma}}_h, \widetilde{\boldsymbol{u}}_h^{\ast}; \boldsymbol{\tau}_h, \boldsymbol{v}_h^{\ast})}{\|\boldsymbol{\tau}_h\|_{0,h} + |\boldsymbol{v}_h^{\ast}|_{1,h}},
\]
which together with \eqref{eq:temp4post} implies
\begin{equation}\label{eq:temp7post}
|\widetilde{\boldsymbol{u}}_h^{\ast}-\boldsymbol{Q}_h\widetilde{\boldsymbol{u}}_h^{\ast}|_{1,h}\lesssim \sup_{(\boldsymbol{\tau}_h, \boldsymbol{v}_h^{\ast})\in \hat{\boldsymbol{\Sigma}}_{h}\times \boldsymbol{V}_{h}^{\ast}}\frac{\mathbb A_h(\widetilde{\boldsymbol{\sigma}}_h, \widetilde{\boldsymbol{u}}_h^{\ast}; \boldsymbol{\tau}_h, \boldsymbol{v}_h^{\ast})}{\|\boldsymbol{\tau}_h\|_{0,h} + |\boldsymbol{v}_h^{\ast}|_{1,h}}.
\end{equation}
Finally we can finish the proof by combining \eqref{eq:temp5post} and \eqref{eq:temp7post}.
\end{proof}

\begin{theorem}
Assume that $\boldsymbol{\sigma}\in\boldsymbol{H}^{k+1}(\Omega; \mathbb{S})$ and $\boldsymbol{u}\in\boldsymbol{H}^{k+2}(\Omega; \mathbb{R}^n)$, then
\begin{equation}\label{eq:errorestimatepost}
\|\boldsymbol{\sigma}-\boldsymbol{\sigma}_h\|_{0,h} + |\boldsymbol{u}-\boldsymbol{u}_h^{\ast}|_{1,h}\lesssim h^{k+1}\left(\|\boldsymbol{\sigma}\|_{k+1}+\|\boldsymbol{u}\|_{k+2}\right).
\end{equation}
Moreover, when $\Omega$ is convex, we have
\begin{equation}\label{eq:errorestimatepost2}
\|\boldsymbol{u}-\boldsymbol{u}_h^{\ast}\|_{0}
\lesssim h^{k+2}\left(\|\boldsymbol{\sigma}\|_{k+1}+\|\boldsymbol{u}\|_{k+2}\right).
\end{equation}
\end{theorem}
\begin{proof}
By direct computation, we have
\[
\mathbb A_h(\boldsymbol{\sigma}, \boldsymbol{u}; \boldsymbol{\tau}_h, \boldsymbol{v}_h^{\ast})=-(\boldsymbol{f}, \boldsymbol{v}_h^{\ast}) \quad  \forall~(\boldsymbol{\tau}_h, \boldsymbol{v}_h^{\ast})\in \boldsymbol{\Sigma}_{h}\times \boldsymbol{V}_{h}^{\ast}.
\]
Combining with \eqref{postmfem}, we get the error equation
\begin{equation}\label{eq:erroreqnpost}
\mathbb A_h(\boldsymbol{\sigma}-\boldsymbol{\sigma}_h, \boldsymbol{u}-\boldsymbol{u}_h^{\ast}; \boldsymbol{\tau}_h, \boldsymbol{v}_h^{\ast})=(\boldsymbol{Q}_h\boldsymbol{f}-\boldsymbol{f}, \boldsymbol{v}_h^{\ast}) \quad  \forall~(\boldsymbol{\tau}_h, \boldsymbol{v}_h^{\ast})\in \boldsymbol{\Sigma}_{h}\times \boldsymbol{V}_{h}^{\ast}.
\end{equation}
Let $\boldsymbol{Q}_h^{\ast}$ be the $L^2$ orthogonal projection from $\boldsymbol{L}^2(\Omega; \mathbb{R}^n)$ onto $\boldsymbol{V}_{h}^{\ast}$.
It holds from \eqref{eq:opcommutative} that
\begin{align*}
&\mathbb A_h(\boldsymbol{\sigma}-\boldsymbol{I}_h^{HZ}\boldsymbol{\sigma}, \boldsymbol{u}-\boldsymbol{Q}_h^{\ast}\boldsymbol{u}; \boldsymbol{\tau}_h, \boldsymbol{v}_h^{\ast}) \\
=&a(\boldsymbol{\sigma}-\boldsymbol{I}_h^{HZ}\boldsymbol{\sigma}, \boldsymbol{\tau}_h) + b(\boldsymbol{\sigma}-\boldsymbol{I}_h^{HZ}\boldsymbol{\sigma}, \boldsymbol{v}_h^{\ast}) \\
&+ (\boldsymbol{\varepsilon}_h(\boldsymbol{u}-\boldsymbol{Q}_h^{\ast}\boldsymbol{u})-\mathfrak{A}(\boldsymbol{\sigma}-\boldsymbol{I}_h^{HZ}\boldsymbol{\sigma}), \boldsymbol{\varepsilon}_h(\boldsymbol{v}_h^{\ast}-\boldsymbol{Q}_h\boldsymbol{v}_h^{\ast})) \\
=&a(\boldsymbol{\sigma}-\boldsymbol{I}_h^{HZ}\boldsymbol{\sigma}, \boldsymbol{\tau}_h) + (\boldsymbol{Q}_h\boldsymbol{f}-\boldsymbol{f}, \boldsymbol{v}_h^{\ast}) \\
&+ (\boldsymbol{\varepsilon}_h(\boldsymbol{u}-\boldsymbol{Q}_h^{\ast}\boldsymbol{u})-\mathfrak{A}(\boldsymbol{\sigma}-\boldsymbol{I}_h^{HZ}\boldsymbol{\sigma}), \boldsymbol{\varepsilon}_h(\boldsymbol{v}_h^{\ast}-\boldsymbol{Q}_h\boldsymbol{v}_h^{\ast})).
\end{align*}
Then we obtain from \eqref{eq:erroreqnpost}, Cauchy-Schwarz inequality and the error estimates of $\boldsymbol{I}_h^{HZ}, \boldsymbol{Q}_h^{\ast}$ and $\boldsymbol{Q}_h$
\begin{align*}
&\mathbb A_h(\boldsymbol{I}_h^{HZ}\boldsymbol{\sigma}-\boldsymbol{\sigma}_h, \boldsymbol{Q}_h^{\ast}\boldsymbol{u}-\boldsymbol{u}_h^{\ast}; \boldsymbol{\tau}_h, \boldsymbol{v}_h^{\ast}) \\
=&a(\boldsymbol{\sigma}-\boldsymbol{I}_h^{HZ}\boldsymbol{\sigma}, \boldsymbol{\tau}_h) + b(\boldsymbol{\sigma}-\boldsymbol{I}_h^{HZ}\boldsymbol{\sigma}, \boldsymbol{v}_h^{\ast}) \\
&+ (\boldsymbol{\varepsilon}_h(\boldsymbol{u}-\boldsymbol{Q}_h^{\ast}\boldsymbol{u})-\mathfrak{A}(\boldsymbol{\sigma}-\boldsymbol{I}_h^{HZ}\boldsymbol{\sigma}), \boldsymbol{\varepsilon}_h(\boldsymbol{v}_h^{\ast}-\boldsymbol{Q}_h\boldsymbol{v}_h^{\ast})) \\
=&a(\boldsymbol{I}_h^{HZ}\boldsymbol{\sigma}-\boldsymbol{\sigma}, \boldsymbol{\tau}_h)- (\boldsymbol{\varepsilon}_h(\boldsymbol{u}-\boldsymbol{Q}_h^{\ast}\boldsymbol{u})-\mathfrak{A}(\boldsymbol{\sigma}-\boldsymbol{I}_h^{HZ}\boldsymbol{\sigma}), \boldsymbol{\varepsilon}_h(\boldsymbol{v}_h^{\ast}-\boldsymbol{Q}_h\boldsymbol{v}_h^{\ast})) \\
\lesssim & h^{k+1}\left (\|\boldsymbol{\sigma}\|_{k+1}+\|\boldsymbol{u}\|_{k+2}\right )(\|\boldsymbol{\tau}_h\|_0+\|\boldsymbol{\varepsilon}_h(\boldsymbol{v}_h^{\ast})\|_0).
\end{align*}
Applying the inf-sup condition \eqref{eq:infsup2post}, it follows
\[
\|\boldsymbol{I}_h^{HZ}\boldsymbol{\sigma}-\boldsymbol{\sigma}_h\|_{0,h} + |\boldsymbol{Q}_h^{\ast}\boldsymbol{u}-\boldsymbol{u}_h^{\ast}|_{1,h}\lesssim h^{k+1}\left (\|\boldsymbol{\sigma}\|_{k+1}+\|\boldsymbol{u}\|_{k+2}\right ).
\]
Hence we will achieve \eqref{eq:errorestimatepost} by using the triangle inequality, and the error estimates of $\boldsymbol{I}_h^{HZ}$ and $\boldsymbol{Q}_h^{\ast}$.

When $\Omega$ is convex, we have from the triangle inequality, the error estimate of $\boldsymbol{Q}_h$ and \eqref{postprocess1}
\begin{align*}
\|\boldsymbol{u}-\boldsymbol{u}_h^{\ast}\|_{0}\leq & \|(\boldsymbol{I}-\boldsymbol{Q}_h)(\boldsymbol{u}-\boldsymbol{u}_h^{\ast})\|_{0}+\|\boldsymbol{Q}_h\boldsymbol{u}-\boldsymbol{Q}_h\boldsymbol{u}_h^{\ast}\|_{0} \\
\lesssim & h|\boldsymbol{u}-\boldsymbol{u}_h^{\ast}|_{1,h}+\|\boldsymbol{Q}_h\boldsymbol{u}-\boldsymbol{u}_h\|_{0}.
\end{align*}
Finally \eqref{eq:errorestimatepost2} is achieved by using \eqref{eq:errorestimatepost} and \eqref{eq:errorestimate3}.
\end{proof}


\section{Block Diagonal and Triangular Preconditioners}
Direct use of the mesh dependent norm $\|\cdot\|_{0,h}\times |\cdot|_{1,h}$ would require the additional assembling of the jump term. In this section, we  first derive equivalent matrix forms for these mesh dependent norms and then construct block-diagonal and block-triangular preconditioners.

\subsection{Equivalent matrix forms of the mesh dependent norms}
By the trace theorem and the inverse inequality, it is easy to see that
\begin{equation}\label{sigmamass}
\|\boldsymbol{\tau}_h\|_{0,h} \eqsim \|\boldsymbol{\tau}_h\|_{0}\quad \forall~\boldsymbol{\tau}_h\in \boldsymbol{\Sigma}_{h},
\end{equation}
which implies that we can use the weighted mass matrix $M_h^{\lambda}$ with $\lambda = 0$, i.e., $M_h$.

For each $\bs v_h\in \bs V_h$, denote by $\boldsymbol{\underline{v_h}}$ the matrix representation of $\bs v_h$ based on the basis of $\bs V_h$ used to form the mass matrix $M_{u, h}$ (cf. \cite[Subsection 4.4]{Xu1992}).
For the mesh dependent norm $|\cdot|_{1,h}$ of displacement, we can use the Schur complement of $(1,1)$ block, i.e., $S_h : = B_hM_h^{-1}B_h^T + C_h$. It is easy to see $S_h$ is SPD and induce a norm $\|\cdot \|_{S_h}$ on $\bs V_h$, i.e.
\[
\|\bs v_h\|_{S_h}^2:=\boldsymbol{\underline{v_h}}^TS_h\boldsymbol{\underline{v_h}}, \quad \forall~\bs v_h\in\bs V_h.
\]
\begin{lemma}\label{lm:S}
 We have the norm equivalence:
 $$
 |\bs v_h|_{1,h} \eqsim \|\bs v_h \|_{S_h}\quad \forall ~\bs v_h \in \bs V_h.
 $$
\end{lemma}
\begin{proof}
We focus on the case $k\geq n+1$ first. The low order case $1\leq k\leq n$ can be proved similarly by adding the stabilization term.

The inf-sup condition \eqref{eq:infsup2} implies $B_h^T$ is injective and thus $S_h$ is SPD and defines an inner product on $\bs V_h$. The identity
\begin{equation}\label{Sh}
(\boldsymbol{\underline{v_h}}^TS_h\boldsymbol{\underline{v_h}})^{1/2} = \|M_h^{-1/2}B_h^T \boldsymbol{\underline{v_h}}\| = \sup_{\boldsymbol{\tau}_h \in \boldsymbol{\Sigma}_{h}} \frac{b(\boldsymbol{\tau}_h, \boldsymbol{v}_h)}{\|\boldsymbol{\tau}_h\|_{0}}, \quad \forall ~\bs v_h \in \bs V_h
\end{equation}
follows from the Riesz representation. Here $\|\cdot\|$ denotes the Euclidean norm of a vector. The inequality $ |\bs v_h|_{1,h} \lesssim \|\bs v_h \|_{S_h}$ is a combination of \eqref{eq:infsup2},  \eqref{sigmamass}, and \eqref{Sh}.

From integration by parts, we can easily get $b(\boldsymbol{\tau}_h, \boldsymbol{v}_h) \lesssim \|\boldsymbol{\tau}_h\|_{0,h}|\bs v_h|_{1,h}$. Then the inequality  $\|\bs v_h \|_{S_h} \lesssim  |\bs v_h|_{1,h}$ follows from \eqref{sigmamass} amd \eqref{Sh}.
\end{proof}

We define the operator $\mcal P_{h}: \bs \Sigma_{h}'\times \bs V_{h}' \rightarrow \bs \Sigma_{h} \times \bs V_{h}$ with the matrix representation
\begin{equation}
\mcal P_{h}
=
\begin{pmatrix}
 M_h^{-1} &0 \\
0& S_{h}^{-1}
\end{pmatrix},
\end{equation}
and denoted by
$$
\mathcal L_h^{\lambda} =
\begin{pmatrix}
M_h^{\lambda} & B_h^T\\
B_h & -C_h
\end{pmatrix}.
$$

\begin{theorem}
The $\mcal P_{h}$ is a uniform preconditioner for $\mcal L_{h}^\lambda$, i.e., the corresponding operator norms
\begin{align*}
\|\mcal P_{h}\mcal L_{h}^{\lambda}\|_{\bs \Sigma _{h}\times \bs V_{h} \to \bs \Sigma _{h}\times \bs V_{h}}, \|(\mcal P_{h}\mcal L_{h}^{\lambda})^{-1}\|_{\bs \Sigma _{h}\times \bs V_{h} \to \bs \Sigma _{h}\times \bs V_{h}}
\end{align*}
are bounded and independent of parameters $h$ and $\lambda$.
\end{theorem}

The mass matrix $M_h^{-1}$ can be further replaced by the inverse of the diagonal matrix or symmetric Gauss-Seidel iteration and thus the computation of $M_h^{-1}$ is not a problem. The difficulty is the inverse of the Schur complement which will be further preconditioned by an auxiliary space preconditioner in the next section.

%

\subsection{Triangular Preconditioner}
When the diagonal of the mass matrix $D_h$ is used, we can make use of the block decomposition
\begin{equation}
\begin{pmatrix}
D_h & B_h^T\\
B_h & -C_h
\end{pmatrix}
\begin{pmatrix}
 I   & D_h^{-1} B_h^T \\
0  &   - I
\end{pmatrix}
=
\begin{pmatrix}
D_h    &   0  \\
B_h   &  \tilde S_h
\end{pmatrix},
\end{equation}
where $\tilde S_h = B_hD_h^{-1}B_h^T+C_h$ to obtain a triangular preconditioner.

We define the operator $\mathcal G_h: \bs \Sigma_{h}'\times \bs V_{h}' \rightarrow \bs \Sigma_{h} \times \bs V_{h}$
\begin{equation}\label{eq:Matrix_Pre}
\mathcal G_h =
\begin{pmatrix}
 I   &  D_h^{-1} B_h^T \\
0  &    -I
\end{pmatrix}
\begin{pmatrix}
D_h    &   0  \\
B_h   &  \tilde S_h
\end{pmatrix}^{-1},
\end{equation}

If we denote by
$$
\widetilde{\mathcal L}_h =
\begin{pmatrix}
D_h & B_h^T\\
B_h & -C_h
\end{pmatrix},
$$
it is trivial to verify that $\mathcal G_h= \widetilde{\mcal L_{h}}^{-1}$. For mass matrix $M_h$, by standard scaling argument, we have $D_h$ is spectrally equivalent to $M_h$ and so $\widetilde{\mcal L_{h}}$ is also stable in the mesh dependent norm. We thus obtain the following result. Detailed eigenvalue analysis of the preconditioned system can be found in~\cite{BankWelfertYserentant1990}.

\begin{theorem}
The $\mcal G_{h}$ is a uniform preconditioner for $\mcal L_{h}^{\lambda}$ i.e., the corresponding operator norms
\begin{align*}
\|\mcal G_{h}\mcal L_{h}^{\lambda}\|_{\bs \Sigma _{h}\times \bs V_{h} \to \bs \Sigma _{h}\times \bs V_{h}}, \|(\mcal G_{h}\mcal L_{h}^{\lambda})^{-1}\|_{\bs \Sigma _{h}\times \bs V_{h} \to \bs \Sigma _{h}\times \bs V_{h}}
\end{align*}
are bounded and independent of parameters $h$ and $\lambda$.
\end{theorem}

In both diagonal and triangular preconditioners, to be practical, we do not compute $S_h^{-1}$ or $\tilde S_h^{-1}$. Instead we shall apply the fast auxiliary space preconditioner to be developed in the next section.


\section{Auxiliary Space Preconditioner}
In this section we  first review the framework on auxiliary space preconditioners developed by Xu \cite{Xu1996} and then construct one for the linear elasticity problem in mixed forms. We use $H^1$ conforming linear element and primary formulation of linear elasticity with $\lambda = 0$ as the auxiliary space preconditioner and verify all assumptions needed in the framework.

\subsection{Framework}

%

Let
\[
\boldsymbol{\mathcal{V}}_{h}:=\left\{\boldsymbol{v}\in \boldsymbol{H}_0^1(\Omega; \mathbb{R}^n): \boldsymbol{v}|_K\in \boldsymbol{P}_{1}(K; \mathbb{R}^n)\quad \forall\,K\in\mathcal
{T}_h\right\}.
\]
Then $\boldsymbol{\mathcal{V}}_h\subset \boldsymbol{V}_h$ for $k\geq2$, and
\begin{equation}\label{eq:temp8}
|\boldsymbol{v}_h|_{1,h} =\|\boldsymbol{\varepsilon}(\boldsymbol{v}_h)\|_0 \eqsim |\boldsymbol{v}_h|_1 \quad \forall~ \boldsymbol{v}_h\in \boldsymbol{\mathcal{V}}_h.
\end{equation}
The conforming linear finite element method for the linear elasticity with $\lambda = 0 $ is defined as follows:
Find $\boldsymbol{u}_h\in \boldsymbol{\mathcal{V}}_{h}$ such that
\[
2\mu (\boldsymbol{\varepsilon}(\boldsymbol{u}_h), \boldsymbol{\varepsilon}(\boldsymbol{v}_h))=(\boldsymbol{f}, \boldsymbol{v}_h) \quad\quad \forall\,\boldsymbol{v}_h\in \boldsymbol{\mathcal{V}}_{h}.
\]
Denote $\mathcal A: \boldsymbol{\mathcal{V}}_{h}\to\boldsymbol{\mathcal{V}}_{h}$ by
\[
(\mathcal A\boldsymbol{w}_h, \boldsymbol{v}_h):=2\mu (\boldsymbol{\varepsilon}(\boldsymbol{w}_h), \boldsymbol{\varepsilon}(\boldsymbol{v}_h)) \quad \forall~\boldsymbol{w}_h, \boldsymbol{v}_h\in \boldsymbol{\mathcal{V}}_{h}.
\]
It is apparent that the operator $\mathcal A$ is SPD.

In what follows we assume $\mathcal{T}_h$ is quasi-uniform.
%
Based on the norm equivalence \eqref{eq:temp8}, we can easily derive the estimate of spectral radius and condition number of the Schur complement operator $S$
\begin{equation}\label{eq:spectralA}
\rho_S=\lambda_{\max}(S) \eqsim h^{-2}, \quad \kappa(S) = \frac{\lambda_{\max}(S)}{\lambda_{\min}(S)}\eqsim h^{-2}.
\end{equation}
The relation between $S$ and $S_h$ is given by
\[
S_h=M_{u,h}\underline{S}
\]
with $\underline{S}$ being the matrix representation of $S$.

We introduce the auxiliary space preconditioner for the Schur complement. The idea is to construct a multigrid method
using $\boldsymbol{V}_h$ as the ``fine'' space and $\boldsymbol{\mathcal{V}}_h$ as the ``coarse''
space. Denote $\mathcal{B}:\: \boldsymbol{\mathcal{V}}_h\rightarrow \boldsymbol{\mathcal{V}}_h$ to be such a ``coarse''
solver. It can be either an exact solver or an approximate solver
that satisfies certain conditions, which will be given later. Next,
on the fine space, we need a smoother $R:\: \boldsymbol{V}_h\rightarrow \boldsymbol{V}_h$,
which is symmetric and positive definite. For example, $R$ can be a
Jacobi or symmetric Gauss-Seidel smoother. Finally, to connect the
``coarse'' space with the ``fine'' space, we need a ``prolongation''
operator $\Pi:\:\boldsymbol{\mathcal{V}}_h\rightarrow \boldsymbol{V}_h$. A ``restriction'' operator
$\Pi^t:\: \boldsymbol{V}_h \rightarrow \boldsymbol{\mathcal{V}}_h$ is consequently defined by
$$
( \Pi^t \boldsymbol{v}, \,\boldsymbol{w}) =(\boldsymbol{v},\, \Pi \boldsymbol{w}) \quad\textrm{for } \boldsymbol{v}\in \boldsymbol{V}_h\textrm{ and } \boldsymbol{w}\in \boldsymbol{\mathcal{V}}_h.
$$
It is also well-known that
the matrix representation of the restriction operator $\Pi^t$ is just the transpose of the
matrix representation of the prolongation operator $\Pi$.
Then, the auxiliary space preconditioner $X:\: \boldsymbol{V}_h\rightarrow \boldsymbol{V}_h$, following the definition in~\cite{Xu1996},
is given by
\begin{align}
&\textrm{Additive} \qquad &&X = R + \Pi \mathcal{B} \Pi^t,  \label{addB}\\
&\textrm{Multiplicative} \qquad && I-XS = (I-R^tS)(I-\Pi \mathcal{B} \Pi^tS)(I-RS). \label{mulB}
\end{align}


According to~\cite{Xu1996}, the following theorem holds.
\begin{theorem}[Xu \cite{Xu1996}] \label{thm:abscond}
  Assume that for all $\boldsymbol{v}\in \boldsymbol{V}_h$, $\boldsymbol{w}\in \boldsymbol{\mathcal{V}}_h$,
  \begin{align}
 (S\boldsymbol{v},\, \boldsymbol{v}) \lesssim (R^{-1} \boldsymbol{v}, \, \boldsymbol{v}) &\lesssim  \rho_S (\boldsymbol{v},\, \boldsymbol{v}), \label{eq:mg-R}\\
    (\mathcal{A} \boldsymbol{w},\, \boldsymbol{w}) \lesssim (\mathcal{B}\mathcal{A} \boldsymbol{w},\, \mathcal{A} \boldsymbol{w}) &\lesssim (\mathcal{A} \boldsymbol{w},\, \boldsymbol{w}), \label{eq:mg-cB}\\
    |\Pi \boldsymbol{w}|_{1,h} &\lesssim |\boldsymbol{w}|_{1} \quad\quad\textrm{(stability of $\Pi$)}, \label{eq:mg-Pi}
  \end{align}
and furthermore, assume that there exists a linear operator $P:\, \boldsymbol{V}_h\rightarrow \boldsymbol{\mathcal{V}}_h$ such that
\begin{align}
  |P \boldsymbol{v}|_{1} &\lesssim |\boldsymbol{v}|_{1,h} \quad\qquad\textrm{(stability of $P$)},\label{eq:mg-P1}\\
  \|\boldsymbol{v}-\Pi P \boldsymbol{v}\|_{0}^2 &\lesssim \rho_S^{-1} |\boldsymbol{v}|_{1,h}^2 \qquad\textrm{(approximability)}.\label{eq:mg-P2}
\end{align}
Then the preconditioner $X$ defined in \eqref{addB} or \eqref{mulB} satisfies
$$
\kappa (XS) \lesssim 1.
$$
\end{theorem}

\subsection{Construction}
Now we construct an auxiliary space preconditioner which satisfies all conditions in Theorem \ref{thm:abscond},
namely, inequalities \eqref{eq:mg-R}-\eqref{eq:mg-P2}.
It is straight forward to pick $\mathcal{B}$ that satisfies condition \eqref{eq:mg-cB}.
For example, $\mathcal{B}$ can be either the direct solver, for which $\mathcal{B}\sim \mathcal{A}^{-1}$,
or one step of classical multigrid iteration which satisfies condition \eqref{eq:mg-cB}.

The smoother $R$ is also easy to define. A Jacobi or a symmetric Gauss-Seidel smoother~\cite{BramblePasciak1992} will satisfy condition \eqref{eq:mg-R}. The operator $\Pi$ is the natural inclusion for $k\geq 2$ and the $L^2$ projection $\boldsymbol{Q}_h$ for $k=1$, i.e., taking the averaging of nodal values inside each simplex.
Then the condition~\eqref{eq:mg-Pi} follows from \eqref{eq:temp8} and \eqref{eq:temp9} immediately.

The technical part is to define an operator $P: \boldsymbol{V}_h \to \boldsymbol{\mathcal{V}}_h$ that satisfy the conditions \eqref{eq:mg-P1}-\eqref{eq:mg-P2}.
%
Note that operator $P$ is needed only in the theoretical analysis.
In the implementation, one needs $\mathcal B$, $R$ and $\Pi$ only.

Construction of $P$ is equivalent to specify the function values at each vertex. For an interior vertex $\boldsymbol{x}_i$ of $\mcal T_h$, denoted by $\Omega_i$ the vertex patch of $\boldsymbol{x}_i$, we will simply choose $(P \boldsymbol{v})(\boldsymbol{x}_i):= |\Omega_i|^{-1}\int_{\Omega_i} \boldsymbol{v} \dx$, i.e., the average of a discontinuous polynomial $\boldsymbol{v}$ in the vertex patch. For boundary vertex $\bs x_i\in \partial \Omega$, we set $(P \boldsymbol{v})(\boldsymbol{x}_i): = 0$.

For any $K\in\mathcal{T}_h$, let $\boldsymbol{Q}_K^0\boldsymbol{v}:=(\boldsymbol{Q}_h^0\boldsymbol{v})|_K=|K|^{-1}\int_{K} \boldsymbol{v}\dx$.
Define
\[
\mathcal{T}_{h,i}:=\{K\in\mathcal{T}_h: K\subset \Omega_i\}, \quad\mathcal{F}_{h,i}:=\{F\in\mathcal{F}_h: \boldsymbol{x}_i\in F\}.
\]
Obviously for interior nodes we have
\begin{equation}\label{eq:Pxi}
(P \boldsymbol{v})(\boldsymbol{x}_i)=\sum_{K\in\mathcal{T}_{h,i}}\frac{|K|}{|\Omega_i|}\boldsymbol{Q}_K^0\boldsymbol{v}.
\end{equation}

The error estimate of the operator $P$ can be derived by standard argument used in \cite{Wang2001, Brenner2004a, BrennerWangZhao2004, HuangHuang2011}.
For completeness, we show it in details as follows.
\begin{lemma}\label{lem:temp3}
The operator $P$ satisfies
 $$
 \| \boldsymbol{v} - P \boldsymbol{v}\|_0 + h|P \boldsymbol{v}|_{1} \lesssim h|\boldsymbol{v}|_{1,h}  \quad \forall~\boldsymbol{v}\in \boldsymbol{V}_h.
 $$
\end{lemma}
\begin{proof}
According to \eqref{eq:Pxi}, it holds for each interior node $\boldsymbol{x}_i$
\[
|\boldsymbol{Q}_K^0\boldsymbol{v}-(P \boldsymbol{v})(\boldsymbol{x}_i)|^2\lesssim \sum_{K^{\prime}\in\mathcal{T}_{h,i}}|\boldsymbol{Q}_K^0\boldsymbol{v}-\boldsymbol{Q}_{K^{\prime}}^0\boldsymbol{v}|^2\lesssim \sum_{F\in\mathcal{F}_{h,i}}|[\boldsymbol{Q}_h^0\boldsymbol{v}]|^2.
\]
For each boundary node $\boldsymbol{x}_i$, we obtain by similar technique and the definition of jump on the boundary
\[
|\boldsymbol{Q}_K^0\boldsymbol{v}-(P \boldsymbol{v})(\boldsymbol{x}_i)|^2=|\boldsymbol{Q}_K^0\boldsymbol{v}|^2\lesssim
\sum_{F\in\mathcal{F}_{h,i}}|[\boldsymbol{Q}_h^0\boldsymbol{v}]|^2.
\]
Then using the scaling argument, we have
\begin{align*}
\sum_{K\in\mathcal{T}_h}h_K^{-2}\|\boldsymbol{Q}_K^0\boldsymbol{v}-P \boldsymbol{v}\|_{0,K}^2=&\sum_{K\in\mathcal{T}_h}\sum_{i=0}^nh_K^{n-2}|\boldsymbol{Q}_K^0\boldsymbol{v}-(P \boldsymbol{v})(\boldsymbol{x}_{K,i})|^2 \\
& \lesssim  \sum_{F\in\mathcal{F}_h}h_F^{-1}\|[\boldsymbol{Q}_h^0\boldsymbol{v}]\|_{0,F}^2.
\end{align*}
From the $L^2$ error estimate \eqref{eq:L2projErrorEstimate}, discrete KornÕs inequality \eqref{eq:korn}, and the norm equivalence \eqref{eq:discreteequivnorm}, we get
\begin{align}
\sum_{K\in\mathcal{T}_h}h_K^{-2}\| \boldsymbol{v} - P \boldsymbol{v}\|_{0,K}^2& \lesssim \sum_{K\in\mathcal{T}_h}h_K^{-2}\|\boldsymbol{v}-\boldsymbol{Q}_K^0\boldsymbol{v}\|_{0,K}^2 + \sum_{K\in\mathcal{T}_h}h_K^{-2}\|\boldsymbol{Q}_K^0\boldsymbol{v}-P \boldsymbol{v}\|_{0,K}^2 \notag\\
& \lesssim  |\boldsymbol{v}|_{1,h}^2 + \sum_{F\in\mathcal{F}_h}h_F^{-1}\|[\boldsymbol{Q}_h^0\boldsymbol{v}]\|_{0,F}^2 \notag\\
& \lesssim  |\boldsymbol{v}|_{1,h}^2 + \sum_{F\in\mathcal{F}_h}h_F^{-1}\|[\boldsymbol{Q}_h^0\boldsymbol{v}-\boldsymbol{v}]\|_{0,F}^2\lesssim |\boldsymbol{v}|_{1,h}^2. \label{eq:temp10}
\end{align}
It follows from \eqref{eq:temp10} and \eqref{eq:L2projErrorEstimate}
\begin{align*}
|P \boldsymbol{v}|_{1}^2 & =\sum_{K\in\mathcal{T}_h}|P \boldsymbol{v}-\boldsymbol{Q}_K^0\boldsymbol{v}|_{1,K}^2\lesssim\sum_{K\in\mathcal{T}_h}h_K^{-2}\|P \boldsymbol{v}-\boldsymbol{Q}_K^0\boldsymbol{v}\|_{0,K}^2 \\
& \lesssim  \sum_{K\in\mathcal{T}_h}h_K^{-2}\|\boldsymbol{v}-P \boldsymbol{v}\|_{0,K}^2 + \sum_{K\in\mathcal{T}_h}h_K^{-2}\|\boldsymbol{v}-\boldsymbol{Q}_K^0\boldsymbol{v}\|_{0,K}^2\lesssim |\boldsymbol{v}|_{1,h}^2.
\end{align*}
Therefore we can finish the proof by combining the last two inequalities.
\end{proof}

\begin{lemma}\label{lem:temp4}
For any $\boldsymbol{v}\in \boldsymbol{V}_h$, it holds
\begin{equation}\label{eq:temp11}
\|\boldsymbol{v}-\Pi P \boldsymbol{v}\|_{0}^2 \lesssim \rho_S^{-1} |\boldsymbol{v}|_{1,h}^2.
\end{equation}
\end{lemma}
\begin{proof}
For $k\geq2$, \eqref{eq:temp11} is the result of Lemma~\ref{lem:temp3} and \eqref{eq:spectralA}. For $k=1$, we obtain from the triangle inequality, \eqref{eq:L2projErrorEstimate}, Lemma~\ref{lem:temp3} and \eqref{eq:spectralA}
\begin{align*}
\|\boldsymbol{v}-\Pi P \boldsymbol{v}\|_{0}^2=&\|\boldsymbol{v}-\boldsymbol{Q}_h P \boldsymbol{v}\|_{0}^2\lesssim \|\boldsymbol{v}- P \boldsymbol{v}\|_{0}^2+\|P\boldsymbol{v}-\boldsymbol{Q}_h P \boldsymbol{v}\|_{0}^2 \\
& \lesssim  \|\boldsymbol{v}- P \boldsymbol{v}\|_{0}^2+h^2|P\boldsymbol{v}|_{1}^2\lesssim h^2|\boldsymbol{v}|_{1,h}^2\lesssim \rho_S^{-1} |\boldsymbol{v}|_{1,h}^2,
\end{align*}
as required.
\end{proof}

Combining Lemma \ref{lm:S}, Theorem~\ref{thm:abscond}, and Lemmas~\ref{lem:temp3}-\ref{lem:temp4}, we have the following estimate of the condition number of $XS$.
\begin{theorem}
Let $R$ be a Jacobi or a symmetric Gauss-Seidel smoother, $\mathcal{B}$ be one step of classical multigrid iteration, and $\Pi$ be $\boldsymbol{Q}_h$. Then the preconditioner $X$ defined in \eqref{addB} or \eqref{mulB} satisfies
$$
\kappa (XS) \lesssim 1.
$$
\end{theorem}


\section{Numerical Results}
In this section, we will report some numerical results to testify the efficiency and robustness of the auxiliary space preconditioners developed in Sections~4-5 for the mixed finite element method~\eqref{stabmfem1}-\eqref{stabmfem2}. Let $\Omega=(-1,1)^2$, $\mu=0.5$ and the load $\boldsymbol{f}=\boldsymbol{1}$. We use the uniform triangulation $\mathcal{T}_h$ of $\Omega$.
The stopping criteria of our iterative methods is the relative residual is less than $10^{-8}$, and the initial guess is zero.
We run the code on the laptop with Intel Core i5 CPU (1.7 GHz) and 4GB RAM.

\subsection{Block Diagonal Preconditioner}

First we use the minimal residual (MINRES) method with the block diagonal preconditioner
\[
\left(
\begin{array}{cc}
D_{h}^{-1} & 0 \\
0 & (B_hD_h^{-1}B_h^T+C_h)^{-1} \\
\end{array}
\right)
\]
to solve the mixed finite element method~\eqref{stabmfem1}-\eqref{stabmfem2}, where $D_h$ is the diagonal matrix of $M_h$.
To solve the Schur complement $B_hD_h^{-1}B_h^T+C_h$, we apply the multiplicative auxiliary space  preconditioner \eqref{mulB}, in which we employ three steps of the Gauss-Seidel smoother for $R$  and one step of V-cycle multigrid method with one pre-smoothing and one post-smoothing for $\mathcal{B}$.

The iteration numbers and CPU time for the block diagonal preconditioned MINRES method are shown in Tables~\ref{tab:diagk1}-\ref{tab:diagk3} for $k=1, 2, 3$, from which we can see that the iteration steps are uniform with respect to the meshsize $h$ and the Lam$\acute{e}$ constant $\lambda$.

\begin{remark}
The iteration steps can be further reduced by introduce a scaling $\textsf{scale}*B_hD_h^{-1}B_h^T$.
\end{remark}

\begin{table}[htbp]
  \centering
  \caption{The iteration steps and CPU time (in seconds) of block diagonal preconditioned MINRES method for $k=1$ }
\begin{tabular}{|c|cc|cc|cc|cc|cc|}
    \hline \hline
     \multirow{2}[4]{*}{$\#$dofs} & \multicolumn{2}{|c|}{$\lambda=0$} & \multicolumn{2}{|c|}{$\lambda=10$} & \multicolumn{2}{|c|}{$\lambda=100$} & \multicolumn{2}{|c|}{$\lambda=1000$} & \multicolumn{2}{|c|}{$\lambda=+\infty$} \\
   \cline{2-11}
          & steps & time & steps & time & steps & time & steps & time & steps & time \\ \hline
    1891         & 43    & 0.08 & 65    & 0.13 & 74    & 0.14  & 74    & 0.14 & 74    & 0.14 \\ \hline
    7363         & 46    & 0.39 & 75    & 0.61 & 84    & 0.69 & 86    & 0.70 & 86    & 0.70 \\ \hline
    29059         & 47    & 1.53 & 78    & 2.48 & 91    & 2.89  & 92    & 3.04 & 92    & 3.04 \\ \hline
    115459         & 47    & 6.11  & 81    & 10.4 & 95    & 12.3 & 96    & 12.3 & 96    & 12.3 \\ \hline
    460291         & 47    & 26.2 & 81    & 45.0 & 97    & 53.6 & 98    & 54.3 & 98    & 54.3 \\
    \hline \hline
    \end{tabular}%
  \label{tab:diagk1}%
\end{table}%

\begin{table}[htbp]
  \centering
  \caption{The iteration steps and CPU time (in seconds) of block diagonal preconditioned MINRES method for $k=2$ }
\begin{tabular}{|c|cc|cc|cc|cc|cc|}
    \hline \hline
    \multirow{2}[4]{*}{$\#$dofs} & \multicolumn{2}{|c|}{$\lambda=0$} & \multicolumn{2}{|c|}{$\lambda=10$} & \multicolumn{2}{|c|}{$\lambda=100$} & \multicolumn{2}{|c|}{$\lambda=1000$} & \multicolumn{2}{|c|}{$\lambda=+\infty$} \\
   \cline{2-11}
           & steps & time & steps & time & steps & time & steps & time & steps & time \\ \hline
    1811         & 57    & 0.17 & 85    & 0.25  & 93    & 0.27 & 94    & 0.28 & 94    & 0.28 \\ \hline
    7075         & 58    & 0.70 & 91    & 1.11  & 98    & 1.19 & 100   & 1.20 & 100   & 1.20 \\ \hline
    27971         & 58    & 2.87 & 93    & 4.58 & 102   & 5.13 & 102   & 5.13 & 102   & 5.13 \\ \hline
    111235         & 58    & 11.8 & 95    & 18.9 & 103   & 20.6 & 104   & 21.4    & 104   & 21.4 \\ \hline
    443651         & 57    & 48.4 & 96    & 79.7 & 104   & 86.0 & 104   & 86.0 & 106   & 87.7 \\
    \hline \hline
    \end{tabular}%
  \label{tab:diagk2}%
\end{table}%

\begin{table}[htbp]
  \centering
  \caption{The iteration steps and CPU time (in seconds) of block diagonal preconditioned MINRES method for $k=3$ }
\begin{tabular}{|c|cc|cc|cc|cc|cc|}
    \hline \hline
    \multirow{2}[4]{*}{$\#$dofs} & \multicolumn{2}{|c|}{$\lambda=0$} & \multicolumn{2}{|c|}{$\lambda=10$} & \multicolumn{2}{|c|}{$\lambda=100$} & \multicolumn{2}{|c|}{$\lambda=1000$} & \multicolumn{2}{|c|}{$\lambda=+\infty$} \\
   \cline{2-11}
           & steps & time & steps & time & steps & time & steps & time & steps & time \\ \hline
    971         & 56    & 0.13 & 89    & 0.19 & 91    & 0.20 & 91    & 0.20 & 91    & 0.20 \\ \hline
    3763         & 58    & 0.58 & 88    & 0.86 & 94    & 0.95 & 94    & 0.95 & 94    & 0.95 \\ \hline
    14819         & 58    & 2.44 & 90    & 3.78 & 96    & 4.01 & 96    & 4.01 & 96    & 4.01 \\ \hline
    58819         & 58    & 9.83 & 90    & 15.3 & 96    & 16.3 & 96    & 16.3 & 97    & 16.4 \\ \hline
    234371         & 57    & 39.7 & 90    & 62.8 & 96    & 66.3 & 98    & 67.4 & 98    & 67.4 \\
    \hline \hline
    \end{tabular}%
  \label{tab:diagk3}%
\end{table}%

\subsection{Block Triangular Preconditioner}

Next we examine the generalized minimal residual (GMRES) method with the block triangular preconditioner
\[
\left(
\begin{array}{cc}
D_{h} & B_h^T \\
B_h & -C_h \\
\end{array}
\right)^{-1} =\left(
\begin{array}{cc}
I & D_{h}^{-1}B_h^T \\
0 & -I \\
\end{array}
\right)
\left(
\begin{array}{cc}
D_{h} & 0 \\
B_h & B_hD_h^{-1}B_h^T+C_h \\
\end{array}
\right)^{-1}.
\]
Set restart=20 in the GMRES method.
We still exploit the same multiplicative auxiliary space  preconditioner as in the block diagonal preconditioner to solve the Schur complement.

The iteration numbers and CPU time for the block triangular preconditioned GMRES method are shown in Tables~\ref{tab:trik1}-\ref{tab:trik3} for $k=1, 2, 3$.
Again the iteration steps are uniform with respect to the meshsize $h$ and the Lam$\acute{e}$ constant $\lambda$. The performance of the block triangular preconditioned GMRES method is better than the block diagonal preconditioned MINRES method. The iteration steps and CPU time are almost halved comparing with the block diagonal preconditioner.

\begin{table}[htbp]
  \centering
  \caption{The iteration steps and CPU time (in seconds) of block triangular preconditioned GMRES method for $k=1$}
\begin{tabular}{|c|cc|cc|cc|cc|cc|}
    \hline \hline
     \multirow{2}[4]{*}{$\#$dofs} & \multicolumn{2}{|c|}{$\lambda=0$} & \multicolumn{2}{|c|}{$\lambda=10$} & \multicolumn{2}{|c|}{$\lambda=100$} & \multicolumn{2}{|c|}{$\lambda=1000$} & \multicolumn{2}{|c|}{$\lambda=+\infty$} \\
   \cline{2-11}
          & steps & time & steps & time & steps & time & steps & time & steps & time \\ \hline
    1891         & 20    & 0.05 & 34    & 0.06 & 38    & 0.08 & 39    & 0.08 & 39    & 0.08 \\ \hline
    7363         & 22    & 0.20 & 39    & 0.36  & 46    & 0.42 & 47    & 0.44 & 47    & 0.44 \\ \hline
    29059         & 24    & 0.88 & 45    & 1.64 & 50    & 1.85 & 51    & 1.88 & 51    & 1.88 \\ \hline
    115459         & 24    & 3.66 & 47    & 7.22 & 54    & 8.12 & 55    & 8.30 & 55    & 8.30 \\ \hline
    460291         & 25    & 16.6 & 50    & 32.6 & 57    & 37.5 & 59    & 39.3 & 59    & 39.3 \\
    \hline \hline
    \end{tabular}%
  \label{tab:trik1}%
\end{table}%

\begin{table}[htbp]
  \centering
  \caption{The iteration steps and CPU time (in seconds) of block triangular preconditioned GMRES method for $k=2$}
\begin{tabular}{|c|cc|cc|cc|cc|cc|}
    \hline \hline
     \multirow{2}[4]{*}{$\#$dofs} & \multicolumn{2}{|c|}{$\lambda=0$} & \multicolumn{2}{|c|}{$\lambda=10$} & \multicolumn{2}{|c|}{$\lambda=100$} & \multicolumn{2}{|c|}{$\lambda=1000$} & \multicolumn{2}{|c|}{$\lambda=+\infty$} \\
   \cline{2-11}
          & steps & time & steps & time & steps & time & steps & time & steps & time \\ \hline
    1811         & 18    & 0.06 & 29    & 0.10 & 31    & 0.11 & 31    & 0.11 & 32    & 0.11 \\ \hline
    7075         & 20    & 0.27 & 32    & 0.45 & 34    & 0.47 & 35    & 0.48 & 35    & 0.48 \\ \hline
    27971         & 22    & 1.24 & 35    & 1.92 & 37    & 2.05 & 38    & 2.12 & 38    & 2.12 \\ \hline
    111235         & 23    & 5.25 & 37    & 8.53  & 40    & 9.23 & 41    & 9.31 & 41    & 9.31 \\ \hline
    443651         & 24    & 23.0 & 39    & 37.1 & 44    & 41.5 & 44    & 41.5 & 44    & 41.5 \\
    \hline \hline
    \end{tabular}%
  \label{tab:trik2}%
\end{table}%

\begin{table}[htbp]
  \centering
  \caption{The iteration steps and CPU time (in seconds) of block triangular preconditioned GMRES method for $k=3$}
\begin{tabular}{|c|cc|cc|cc|cc|cc|}
    \hline \hline
     \multirow{2}[4]{*}{$\#$dofs} & \multicolumn{2}{|c|}{$\lambda=0$} & \multicolumn{2}{|c|}{$\lambda=10$} & \multicolumn{2}{|c|}{$\lambda=100$} & \multicolumn{2}{|c|}{$\lambda=1000$} & \multicolumn{2}{|c|}{$\lambda=+\infty$} \\
   \cline{2-11}
          & steps & time & steps & time & steps & time & steps & time & steps & time \\ \hline
    971         & 20    & 0.05 & 27    & 0.06 & 28    & 0.06 & 28    & 0.06 & 28    & 0.06 \\ \hline
    3763         & 21    & 0.24 & 29    & 0.31 & 30    & 0.33 & 30    & 0.33 & 30    & 0.33 \\ \hline
    14819         & 22    & 1.02 & 30    & 1.36 & 32    & 1.47 & 32    & 1.47 & 32    & 1.47 \\ \hline
    58819         & 23    & 4.30 & 31    & 5.80 & 33    & 6.16 & 33    & 6.16 & 33    & 6.16 \\ \hline
    234371         & 24    & 18.6 & 32    & 24.6 & 34    & 26.1 & 35    & 26.9 & 35    & 26.9 \\
    \hline \hline
    \end{tabular}%
  \label{tab:trik3}%
\end{table}%



\begin{thebibliography}{10}

\bibitem{AdamsCockburn2005}
S.~Adams and B.~Cockburn.
\newblock A mixed finite element method for elasticity in three dimensions.
\newblock {\em J. Sci. Comput.}, 25(3):515--521, 2005.

\bibitem{ArnoldAwanou2005}
D.~N. Arnold and G.~Awanou.
\newblock Rectangular mixed finite elements for elasticity.
\newblock {\em Math. Models Methods Appl. Sci.}, 15(9):1417--1429, 2005.

\bibitem{ArnoldAwanouWinther2008}
D.~N. Arnold, G.~Awanou, and R.~Winther.
\newblock Finite elements for symmetric tensors in three dimensions.
\newblock {\em Math. Comp.}, 77(263):1229--1251, 2008.

\bibitem{ArnoldBrezziMarini2005}
D.~N. Arnold, F.~Brezzi, and L.~D. Marini.
\newblock A family of discontinuous {G}alerkin finite elements for the
  {R}eissner-{M}indlin plate.
\newblock {\em J. Sci. Comput.}, 22/23:25--45, 2005.

\bibitem{ArnoldWinther2002}
D.~N. Arnold and R.~Winther.
\newblock Mixed finite elements for elasticity.
\newblock {\em Numer. Math.}, 92(3):401--419, 2002.

\bibitem{AxelssonPadiy1999}
O.~Axelsson and A.~Padiy.
\newblock On a robust and scalable linear elasticity solver based on a saddle
  point formulation.
\newblock {\em Internat. J. Numer. Methods Engrg.}, 44(6):801--818, 1999.

\bibitem{BankWelfertYserentant1990}
R.~E. Bank, B.~D. Welfert, and H.~Yserentant.
\newblock A class of iterative methods for solving saddle point problems.
\newblock {\em Numer. Math.}, 56(7):645--666, 1990.

\bibitem{BoffiBrezziFortin2013}
D.~Boffi, F.~Brezzi, and M.~Fortin.
\newblock {\em Mixed finite element methods and applications}, volume~44 of
  {\em Springer Series in Computational Mathematics}.
\newblock Springer, Heidelberg, 2013.

\bibitem{BraessVerfurth1996}
D.~Braess and R.~Verf{\"u}rth.
\newblock A posteriori error estimators for the {R}aviart-{T}homas element.
\newblock {\em SIAM J. Numer. Anal.}, 33(6):2431--2444, 1996.

\bibitem{BramblePasciak1992}
J.~H. Bramble and J.~E. Pasciak.
\newblock The analysis of smoothers for multigrid algorithms.
\newblock {\em Math. Comp.}, 58(198):467--488, 1992.

\bibitem{Brenner1993}
S.~C. Brenner.
\newblock A nonconforming mixed multigrid method for the pure displacement
  problem in planar linear elasticity.
\newblock {\em SIAM J. Numer. Anal.}, 30(1):116--135, 1993.

\bibitem{Brenner1994}
S.~C. Brenner.
\newblock A nonconforming mixed multigrid method for the pure traction problem
  in planar linear elasticity.
\newblock {\em Math. Comp.}, 63(208):435--460, S1--S5, 1994.

\bibitem{Brenner2004a}
S.~C. Brenner.
\newblock Korn's inequalities for piecewise {$H^1$} vector fields.
\newblock {\em Math. Comp.}, 73(247):1067--1087, 2004.

\bibitem{BrennerLiSung2014}
S.~C. Brenner, H.~Li, and L.-Y. Sung.
\newblock Multigrid methods for saddle point problems: {S}tokes and {L}am\'e
  systems.
\newblock {\em Numer. Math.}, 128(2):193--216, 2014.

\bibitem{BrennerScott2008}
S.~C. Brenner and L.~R. Scott.
\newblock {\em The mathematical theory of finite element methods}, volume~15 of
  {\em Texts in Applied Mathematics}.
\newblock Springer, New York, third edition, 2008.

\bibitem{BrennerWangZhao2004}
S.~C. Brenner, K.~Wang, and J.~Zhao.
\newblock Poincar\'e-{F}riedrichs inequalities for piecewise {$H\sp 2$}
  functions.
\newblock {\em Numer. Funct. Anal. Optim.}, 25(5-6):463--478, 2004.

\bibitem{BrixCamposDahmen2008}
K.~Brix, M.~Campos~Pinto, and W.~Dahmen.
\newblock A multilevel preconditioner for the interior penalty discontinuous
  {G}alerkin method.
\newblock {\em SIAM J. Numer. Anal.}, 46(5):2742--2768, 2008.

\bibitem{Chen.L2008c}
L.~Chen.
\newblock {$i$FEM}: {An Integrated Finite Element Methods Package in MATLAB}.
\newblock {\em Technical Report, University of California at Irvine}, 2009.

\bibitem{ChenHuHuang2015}
L.~Chen, J.~Hu, and X.~Huang.
\newblock Stabilized mixed finite element methods for linear elasticity on
  simplicial grids in $\mathbb{R}^{n}$.
\newblock {\em arXiv:1512.03998}, 2015.

\bibitem{ChenWangWangYe2015}
L.~Chen, J.~Wang, Y.~Wang, and X.~Ye.
\newblock An auxiliary space multigrid preconditioner for the weak {G}alerkin
  method.
\newblock {\em Comput. Math. Appl.}, 70(4):330--344, 2015.

\bibitem{chen2016multigrid}
L.~Chen, Y.~Wu, L.~Zhong, and J.~Zhou.
\newblock Multigrid preconditioners for mixed finite element methods of vector
  laplacian.
\newblock {\em arXiv preprint arXiv:1601.04095}, 2016.

\bibitem{Ciarlet1978}
P.~G. Ciarlet.
\newblock {\em The finite element method for elliptic problems}.
\newblock North-Holland Publishing Co., Amsterdam, 1978.
\newblock Studies in Mathematics and its Applications, Vol. 4.

\bibitem{CockburnDuboisGopalakrishnanTan2014}
B.~Cockburn, O.~Dubois, J.~Gopalakrishnan, and S.~Tan.
\newblock Multigrid for an {HDG} method.
\newblock {\em IMA J. Numer. Anal.}, 34(4):1386--1425, 2014.

\bibitem{DouglasRoberts1985}
J.~Douglas, Jr. and J.~E. Roberts.
\newblock Global estimates for mixed methods for second order elliptic
  equations.
\newblock {\em Math. Comp.}, 44(169):39--52, 1985.

\bibitem{Hannukainen2012}
A.~Hannukainen.
\newblock Continuous preconditioners for the mixed {P}oisson problem.
\newblock {\em BIT}, 52(1):65--83, 2012.

\bibitem{HiptmairXu2007}
R.~Hiptmair and J.~Xu.
\newblock Nodal auxiliary space preconditioning in {${\bf H}({\bf curl})$} and
  {${\bf H}({\rm div})$} spaces.
\newblock {\em SIAM J. Numer. Anal.}, 45(6):2483--2509, 2007.

\bibitem{HongKrausXuZikatanov2016}
Q.~Hong, J.~Kraus, J.~Xu, and L.~Zikatanov.
\newblock A robust multigrid method for discontinuous {G}alerkin
  discretizations of {S}tokes and linear elasticity equations.
\newblock {\em Numer. Math.}, 132(1):23--49, 2016.

\bibitem{Hu2015a}
J.~Hu.
\newblock Finite element approximations of symmetric tensors on simplicial
  grids in {$\Bbb R^n$}: the higher order case.
\newblock {\em J. Comput. Math.}, 33(3):283--296, 2015.

\bibitem{HuZhang2015c}
J.~Hu and S.~Zhang.
\newblock A family of conforming mixed finite elements for linear elasticity on
  triangular grids.
\newblock {\em arXiv:1406.7457}, 2015.

\bibitem{HuZhang2015}
J.~Hu and S.~Zhang.
\newblock A family of symmetric mixed finite elements for linear elasticity on
  tetrahedral grids.
\newblock {\em Sci. China Math.}, 58(2):297--307, 2015.

\bibitem{HuangHuang2011}
J.~Huang and X.~Huang.
\newblock Local and parallel algorithms for fourth order problems discretized
  by the {M}orley-{W}ang-{X}u element method.
\newblock {\em Numer. Math.}, 119(4):667--697, 2011.

\bibitem{KlawonnStarke2004}
A.~Klawonn and G.~Starke.
\newblock A preconditioner for the equations of linear elasticity discretized
  by the {PEERS} element.
\newblock {\em Numer. Linear Algebra Appl.}, 11(5-6):493--510, 2004.

\bibitem{KolevVassilevski2008}
T.~V. Kolev and P.~S. Vassilevski.
\newblock Auxiliary space {AMG} for {$H$}(curl) problems.
\newblock In {\em Domain decomposition methods in science and engineering
  {XVII}}, volume~60 of {\em Lect. Notes Comput. Sci. Eng.}, pages 147--154.
  Springer, Berlin, 2008.

\bibitem{KolevVassilevski2009}
T.~V. Kolev and P.~S. Vassilevski.
\newblock Parallel auxiliary space {AMG} for {$H({\rm curl})$} problems.
\newblock {\em J. Comput. Math.}, 27(5):604--623, 2009.

\bibitem{KrausLazarovLymberyMargenovEtAl2016}
J.~Kraus, R.~Lazarov, M.~Lymbery, S.~Margenov, and L.~Zikatanov.
\newblock Preconditioning heterogeneous {$H({\rm div})$} problems by additive
  {S}chur complement approximation and applications.
\newblock {\em SIAM J. Sci. Comput.}, 38(2):A875--A898, 2016.

\bibitem{KrausLymberyMargenov2015}
J.~Kraus, M.~Lymbery, and S.~Margenov.
\newblock Auxiliary space multigrid method based on additive {S}chur complement
  approximation.
\newblock {\em Numer. Linear Algebra Appl.}, 22(6):965--986, 2015.

\bibitem{Lee1998}
C.-O. Lee.
\newblock Multigrid methods for the pure traction problem of linear elasticity:
  mixed formulation.
\newblock {\em SIAM J. Numer. Anal.}, 35(1):121--145, 1998.

\bibitem{LeeWuChen2009}
Y.-J. Lee, J.~Wu, and J.~Chen.
\newblock Robust multigrid method for the planar linear elasticity problems.
\newblock {\em Numer. Math.}, 113(3):473--496, 2009.

\bibitem{LiXie2015}
B.~Li and X.~Xie.
\newblock A two-level algorithm for the weak {G}alerkin discretization of
  diffusion problems.
\newblock {\em J. Comput. Appl. Math.}, 287:179--195, 2015.

\bibitem{LovadinaStenberg2006}
C.~Lovadina and R.~Stenberg.
\newblock Energy norm a posteriori error estimates for mixed finite element
  methods.
\newblock {\em Math. Comp.}, 75(256):1659--1674, 2006.

\bibitem{MardalWinther2011}
K.-A. Mardal and R.~Winther.
\newblock Preconditioning discretizations of systems of partial differential
  equations.
\newblock {\em Numer. Linear Algebra Appl.}, 18(1):1--40, 2011.

\bibitem{PasciakWang2006}
J.~E. Pasciak and Y.~Wang.
\newblock A multigrid preconditioner for the mixed formulation of linear plane
  elasticity.
\newblock {\em SIAM J. Numer. Anal.}, 44(2):478--493, 2006.

\bibitem{RustenVassilevskiWinther1996}
T.~Rusten, P.~S. Vassilevski, and R.~Winther.
\newblock Interior penalty preconditioners for mixed finite element
  approximations of elliptic problems.
\newblock {\em Math. Comp.}, 65(214):447--466, 1996.

\bibitem{Schoberl1999}
J.~Sch{\"o}berl.
\newblock Multigrid methods for a parameter dependent problem in primal
  variables.
\newblock {\em Numer. Math.}, 84(1):97--119, 1999.

\bibitem{Stenberg1986}
R.~Stenberg.
\newblock On the construction of optimal mixed finite element methods for the
  linear elasticity problem.
\newblock {\em Numer. Math.}, 48(4):447--462, 1986.

\bibitem{Stenberg1991}
R.~Stenberg.
\newblock Postprocessing schemes for some mixed finite elements.
\newblock {\em RAIRO Mod\'el. Math. Anal. Num\'er.}, 25(1):151--167, 1991.

\bibitem{TuminaroXuZhu2009}
R.~S. Tuminaro, J.~Xu, and Y.~Zhu.
\newblock Auxiliary space preconditioners for mixed finite element methods.
\newblock In {\em Domain decomposition methods in science and engineering
  {XVIII}}, volume~70 of {\em Lect. Notes Comput. Sci. Eng.}, pages 99--109.
  Springer, Berlin, 2009.

\bibitem{Wang2001}
M.~Wang.
\newblock On the necessity and sufficiency of the patch test for convergence of
  nonconforming finite elements.
\newblock {\em SIAM J. Numer. Anal.}, 39(2):363--384 (electronic), 2001.

\bibitem{Wang2005}
Y.~Wang.
\newblock Overlapping {S}chwarz preconditioner for the mixed formulation of
  plane elasticity.
\newblock {\em Appl. Numer. Math.}, 54(2):292--309, 2005.

\bibitem{Wieners2000}
C.~Wieners.
\newblock Robust multigrid methods for nearly incompressible elasticity.
\newblock {\em Computing}, 64(4):289--306, 2000.

\bibitem{Xu1992}
J.~Xu.
\newblock Iterative methods by space decomposition and subspace correction.
\newblock {\em SIAM Rev.}, 34(4):581--613, 1992.

\bibitem{Xu1996}
J.~Xu.
\newblock The auxiliary space method and optimal multigrid preconditioning
  techniques for unstructured grids.
\newblock {\em Computing}, 56(3):215--235, 1996.

\bibitem{ZhangXu2014}
S.~Zhang and J.~Xu.
\newblock Optimal solvers for fourth-order {PDE}s discretized on unstructured
  grids.
\newblock {\em SIAM J. Numer. Anal.}, 52(1):282--307, 2014.

\bibitem{ZhongChungLiu2015}
L.~Zhong, E.~T. Chung, and C.~Liu.
\newblock Fast solvers for the symmetric {IPDG} discretization of second order
  elliptic problems.
\newblock {\em Int. J. Numer. Anal. Model.}, 12(3):455--475, 2015.

\bibitem{ZhuSifakisTeranBrandt2010}
Y.~Zhu, E.~Sifakis, J.~Teran, and A.~Brandt.
\newblock An efficient multigrid method for the simulation of high-resolution
  elastic solids.
\newblock {\em ACM Trans. Graph. (TOG)}, 29(2), MAR 2010.

\end{thebibliography}
\end{document}